\UseRawInputEncoding
\documentclass[12pt]{article}

\usepackage{dsfont}
\usepackage{amsfonts,amsmath,amsthm}
\usepackage{algorithm, algorithmic}
\usepackage{color}
\usepackage{graphicx}
\usepackage{stmaryrd}        

\usepackage[paper=a4paper,dvips,top=2cm,left=2cm,right=2cm,
    foot=1cm,bottom=4cm]{geometry}

\begin{document}

\title{Minimax principle for eigenvalues of dual quaternion Hermitian matrices and generalized inverses of dual quaternion matrices}
\author{ Chen Ling\footnote{%
    Department of Mathematics, Hangzhou Dianzi University, Hangzhou, 310018, China
    ({\tt
macling@hdu.edu.cn}). This author's work was supported by Natural Science Foundation of China (No. 11971138).}
    \and
    Liqun Qi\thanks{Department of Applied Mathematics, The Hong Kong Polytechnic University, Hung Hom,
    Kowloon, Hong Kong; Department of Mathematics, Hangzhou Dianzi University, Hangzhou, 310018, China ({\tt maqilq@polyu.edu.hk}).} \and and \
Hong Yan\thanks{Department of Electrical Engineering, and Centre for Intelligent Multidimensional Data Analysis, City University of Hong Kong, Kowloon, Hong Kong ({\tt h.yan@cityu.edu.hk}).  This author's work was supported by Hong Kong Research Grants Council (Project 11204821), Hong Kong Innovation and Technology Commission (InnoHK Project CIMDA) and City University of Hong Kong (Project 9610034). }
}
\date{\today}
\maketitle

\begin{abstract}

 Dual quaternions can represent rigid body motion in 3D spaces, and have found wide applications
in robotics, 3D motion modelling and control, and computer graphics. In this paper, we introduce three different right linear independency concepts for a set of dual quaternion vectors, and study some related basic properties for dual quaternion vectors and dual quaternion matrices. We present a minimax principle for eigenvalues of dual quaternion Hermitian matrices. Based upon a newly established Cauchy-Schwarz inequality for dual quaternion vectors and singular value decomposition of dual quaternion matrices, we propose an inequality for singular values of dual quaternion matrices. Finally, we introduce the concept of generalized inverses of dual quaternion matrices, and present  necessary and sufficient conditions for a dual quaternion matrix to be one of four types of generalized inverses of another dual quaternion matrix.

\medskip



\end{abstract}

\renewcommand{\Re}{\mathds{R}}
\newcommand{\rank}{\mathrm{rank}}
\renewcommand{\span}{\mathrm{span}}
\newcommand{\X}{\mathcal{X}}
\newcommand{\A}{\mathcal{A}}
\newcommand{\I}{\mathcal{I}}
\newcommand{\B}{\mathcal{B}}
\newcommand{\C}{\mathcal{C}}
\newcommand{\OO}{\mathcal{O}}
\newcommand{\e}{\mathbf{e}}
\newcommand{\0}{\mathbf{0}}
\newcommand{\dd}{\mathbf{d}}
\newcommand{\ee}{\mathbf{e}}
\newcommand{\ii}{\mathbf{i}}
\newcommand{\jj}{\mathbf{j}}
\newcommand{\kk}{\mathbf{k}}
\newcommand{\va}{\mathbf{a}}
\newcommand{\vb}{\mathbf{b}}
\newcommand{\vc}{\mathbf{c}}
\newcommand{\vg}{\mathbf{g}}
\newcommand{\vr}{\mathbf{r}}
\newcommand{\vt}{\rm{vec}}
\newcommand{\vx}{\mathbf{x}}
\newcommand{\vy}{\mathbf{y}}
\newcommand{\vu}{\mathbf{u}}
\newcommand{\vv}{\mathbf{v}}
\newcommand{\y}{\mathbf{y}}
\newcommand{\vz}{\mathbf{z}}
\newcommand{\T}{\top}

\newtheorem{Thm}{Theorem}[section]
\newtheorem{Def}[Thm]{Definition}
\newtheorem{Ass}[Thm]{Assumption}
\newtheorem{Lem}[Thm]{Lemma}
\newtheorem{Prop}[Thm]{Proposition}
\newtheorem{Cor}[Thm]{Corollary}
\newtheorem{example}[Thm]{Example}
\newtheorem{ReK}[Thm]{Remark}

\section{Introduction}\label{Introd}
Dual quaternions were introduced by Clifford \cite{Cl73} in 1873, and have become one of the core knowledge of Clifford algebra or geometric algebra. In mathematics, the dual quaternions are an $8$-dimensional real algebra isomorphic to the tensor product of the quaternions and the dual numbers. In mechanics, the dual quaternions are applied as a number system to represent rigid transformations in three dimensions. Similar to the way that rotations in 3D space can be represented by unit quaternions, rigid motions in 3D space can be represented by unit dual quaternions. Therefore, 
dual quaternions have found wide applications in engineering fields, such as 3D computer graphics, robotics control, computer vision, and neuroscience and biomechanics \cite{BLH19,CKJC16,Da99,LLB13,MKO14,MXXLY18,PV10,TRA11,WYL12,WZ14,WDW23}. 

The study of dual quaternion matrices and their applications in formation control in 3D space can be traced back to the Ph.D. thesis of X. Wang \cite{Wx11} in 2011. In an unpublished manuscript, Wang, Yu and Zheng \cite{WYZ} introduced three classes of dual quaternion matrices for studying multi-agent formation control, namely relative configuration adjacency matrices, logarithm adjacency matrices and relative twist adjacency matrices. Recently, Qi and Luo \cite{QL21} showed that dual quaternion Hermitian matrices have very nice spectral properties. They showed that an $n \times n$ dual quaternion Hermitian matrix has exactly $n$ right eigenvalues, which are all dual numbers and are also the left eigenvalues of this Hermitian matrix. Thus, we may simply call them eigenvalues of that Hermitian matrix. With the help of the total order of dual numbers introduced in \cite{QLY21}, we know that a dual quaternion Hermitian matrix is positive semi-definite or positive definite if and only if its $n$ dual number eigenvalues are nonnegative or positive, respectively.  
Moreover, singular value decomposition of qual quaternion matrices was established in \cite{QL21}, and von Neumann type trace inequality for dual quaternion matrices and their low-rank approximations also were studied \cite{LHQ22,LHQF23}. Then, Qi, Wang and Luo \cite{QWL22} showed that the relative configuration adjacency matrix and the logarithm adjacency matrix are dual quaternion Hermitian matrices. These dual quaternion matrices play important roles in multi-agent formation control, and the related spectral and positive semi-definite properties pave the way for us to study stability
issues of the multi-agent formation control problem. The work in \cite{QWL22} establish a bridge for the research on dual quaternion matrices and multi-agent formation control.  Furthermore, Cui and Qi \cite{CQ23} proposed a power method to computing eigenvalues of a dual quaternion Hermitian
matrix, and applied it to the simultaneous location and mapping problem. 

This paper discusses some properties of dual quaternion vectors and dual quaternion matrices, such as the right linear independency of dual quaternion vectors, the minimax principle related with the eigenvalues of dual quaternion Hermitian matrices, and the generalized inverses of dual quaternion matrices, etc.

In the next section, we present some preliminary knowledge on quaternions, dual numbers and dual quaternions. In Section \ref{RighIN-Rank}, we introduce three different right linear independency concepts for dual quaternion vectors. Based upon these definitions, we study some basic properties for dual quaternion vectors and dual quaternion matrices. We present a minimax principle for eigenvalues of dual quaternion Hermitian matrices in Section \ref{MinimaxPrin}. Moreover, relying on a newly established Cauchy-Schwarz inequality for dual quaternion vectors and singular value decomposition of dual quaternion matrices presented in \cite{QL21}, we propose an inequality for singular values of dual quaternion matrices. In Section \ref{GenInverseDQM}, we introduce the concept of generalized inverses of dual quaternion matrices, and present necessary and sufficient conditions for $X=A_{\rm st}^\dag-A_{\rm st}^\dag A_{\rm I}A_{\rm st}^\dag\epsilon$ (defined in Proposition \ref{DQ2I}) to be the $\{1\}$-, $\{3\}$- and $\{4\}$-dual quaternion generalized inverses of a dual quaternion matrix $A$. Conclusions are made in Section \ref{Conclusion}.

\section{Preliminaries}\label{Prelim}
\subsection{Quaternions and quaternion matrices}
Let $\mathbb{R}$ denote the field of the real numbers, and denote by $\mathbb Q$ the four-dimensional vector space of the quaternions over $\mathbb{R}$, with an ordered basis, denoted by $\ee, \ii, \jj$ and $\kk$. A quaternion $q\in \mathbb{Q}$ has the form
$q = q_0\ee + q_1\ii + q_2\jj + q_3\kk$, where $q_0, q_1, q_2$ and $q_3$ are real numbers, $\ii, \jj$ and $\kk$ are three imaginary units of quaternions, satisfying $\ii^2 = \jj^2 = \kk^2 =\ii\jj\kk = -1$, $\ii\jj = -\jj\ii = \kk$, $\jj\kk = - \kk\jj = \ii$, $\kk\ii = -\ii\kk = \jj$.   In the following, we will omit the real unit $\ee$, and simply denote $q \in \mathbb Q$ as $q = q_0 + q_1\ii + q_2\jj + q_3\kk$.
The real part of $q$ is ${\rm Re}(q):= q_0$, 
the imaginary part of $q$ is ${\rm Im}(q):= q_1\ii + q_2\jj +q_3\kk$, the conjugate of $q$ is $\bar q := q_0 -q_1\ii -q_2\jj -q_3\kk$, and the norm of $q$ is $|q|:=\sqrt{\bar qq}=\sqrt{q_0^2+q_1^2+q_2^2+q_3^2}$.
A quaternion is called imaginary if its real part is zero. The multiplication of quaternions satisfies the distribution law, but is noncommutative. In fact, $\mathbb{Q}$ is an associative but non-commutative algebra of four rank over $\mathbb{R}$, called quaternion skew-field.

Denote by $\mathbb{Q}^{m\times n}$  the collection of all $m\times n$ matrices with quaternion entries. Specially, $\mathbb{Q}^{m\times 1}$ is abbreviated as $\mathbb{Q}^m$, which is the collection of quaternion column vectors with $m$ components. We denote the quaternion column vectors by boldfaced lowercase letters (e.g., ${\bf u,v,\ldots}$). 
A matrix $A\in \mathbb{Q}^{m\times n}$ can be written as $A = A_0 + A_1\ii + A_2\jj + A_3\kk$, where $A_0, A_1, A_2$ and $A_3$ are real matrices. For $A=(a_{ij})\in \mathbb{Q}^{m\times n}$ and $b\in \mathbb{Q}$, the right scalar multiplication is defined as $Ab=(a_{ij}b)$. It is easy to know that $\mathbb{Q}^{m\times n}$ is a right vector space over $\mathbb{Q}$ under the addition and the right scalar multiplication \cite{WLZZ18}.  For given $A\in \mathbb{Q}^{m\times n}$, the transpose of $A$ is denoted as $A^\top = (a_{ji})$, the conjugate of $A$ is denoted as ${\bar A} = (\bar a_{ij})$, and the conjugate transpose of $A$ is denoted as $A^*=(\bar a_{ji})=\bar A^\top$. A square matrix $A\in \mathbb{Q}^{m\times m}$ is called nonsingular (invertible) if $AB = BA = I_m$ for some $B\in \mathbb{Q}^{m\times m}$, where $I_m$ stands for the $m\times m$ identity matrix. In that case, we denote $A^{-1} = B$. A square matrix $A\in \mathbb{Q}^{m\times m}$ is called normal if $AA^*=A^*A$, Hermitian if $A^*=A$, and unitary if $A$ is nonsingular and $A^{-1}=A^*$. We have $(AB)^{-1}=B^{-1}A^{-1}$ if $A$ and $B$ are nonsingular, and $(A^*)^{-1}=(A^{-1})^*$ if $A$ is nonsingular. For any $A=(a_{ij})\in \mathbb{Q}^{m\times n}$, the Frobenius norm of $A$ is defined by
$$
\|A\|=\sqrt{\sum_{i=1}^m\sum_{j=1}^n|a_{ij}|^2}.
$$
\begin{Def}\cite{WLZZ18}\label{Def-RIn}
Let ${\bf u}^{(1)},{\bf u}^{(2)},\ldots,{\bf u}^{(n)}\in \mathbb{Q}^m$. We say that $\{{\bf u}^{(1)},{\bf u}^{(2)},\ldots,{\bf u}^{(n)}\}$ is right linearly independent, if for any $k_1,k_2,\ldots,k_n\in \mathbb{Q}$,
$$
{\bf u}^{(1)}k_1+{\bf u}^{(2)}k_2+\ldots+{\bf u}^{(n)}k_n={\bf 0}~~~\Rightarrow~~~k_1=k_2=\ldots=k_n=0.
$$
\end{Def}

From Definition \ref{Def-RIn}, we know that $\{{\bf u}^{(1)},{\bf u}^{(2)},\ldots,{\bf u}^{(n)}\}$ is right linearly independent, if and only if  $A{\bf x}={\bf 0}$ has only a unique zero solution in $\mathbb{Q}^n$, where $A=[{\bf u}^{(1)},{\bf u}^{(2)},\ldots,{\bf u}^{(n)}]$. The rank of the matrix $A\in \mathbb{Q}^{m\times n}$ is defined to be the maximum number of columns of $A$ which are right linearly independent, and denoted by ${\rm rank}(A)$. It is easy to see that ${\rm rank}(A)={\rm rank}(BAC)$ for any nonsingular matrices $B$ and $C$ of suitable size. Thus ${\rm rank}(A)$ is equal to the number of positive singular values of $A$. For given $A\in \mathbb{Q}^{m\times n}$, the solution set of $A{\bf x}={\bf 0}$ form a subspace of $\mathbb{Q}^n$, and it has dimension $r$ if and only if ${\rm rank}(A)=n-r$. Moreover, we know that, for given $A\in \mathbb{Q}^{m\times m}$, $A$ is nonsingular if and only if $A$ is of full rank $m$. For given $A\in \mathbb{Q}^{m\times m}$, if there exist ${\bf x}\in \mathbb{Q}^m\backslash\{\bf 0\}$ and $\lambda\in \mathbb{Q}$, such that $A{\bf x} = {\bf x}\lambda$, then $\lambda$ is called a right eigenvalue of $A$, with ${\bf x}$ as an associated right eigenvector. See \cite{WLZZ18} for more details. From the arguments above, we have 
\begin{Prop}\label{SingularP}
For $A\in \mathbb{Q}^{m\times m}$, $A$ is nonsingular, if and only if $\lambda=0$ is not a right eigenvalue of $A$.
\end{Prop}

For given ${\bf u}=(u_1,u_2,\ldots,u_m)^\top$ and ${\bf v}=(v_1,v_2,\ldots,v_m)^\top$ in $\mathbb{Q}^m$, denote by $\langle {\bf u},{\bf v}\rangle$ the quaternion-valued inner product, i.e., $\langle {\bf u},{\bf v}\rangle={\bf v}^*{\bf u}=\sum_{i=1}^m\bar v_iu_i$. It is easy to see that $\langle {\bf u}, {\bf v} \alpha+{\bf w} \beta\rangle= \bar\alpha\langle {\bf u}, {\bf v}\rangle +\bar\beta\langle{\bf u},{\bf w}\rangle$ and $\langle {\bf u},{\bf v}\rangle=\overline{\langle {\bf v},{\bf u}\rangle}$ for any ${\bf u},{\bf v}, {\bf w}\in \mathbb{Q}^m$ and $\alpha,\beta\in \mathbb{Q}$.

\begin{Prop} \cite{BB17}\label{QCauchy-Inequality}
(Cauchy-Schwarz inequality on $\mathbb{Q}^m$) For any ${\bf u},{\bf v}\in \mathbb{Q}^m$, it holds that
$$
|\langle{\bf u},{\bf v}\rangle|\leq \|{\bf u}\|\|\bf v\|.
$$
\end{Prop}

\begin{Def}\cite{Ro14}\label{Def-Orthog}
Let ${\bf u},{\bf v}\in \mathbb{Q}^m\backslash\{0\}$. We say ${\bf u},{\bf v}$  are orthogonal if $\langle {\bf u},{\bf v}\rangle=0$. An $n$-tuple $\{{\bf u}^{(1)},{\bf u}^{(2)},\ldots, {\bf u}^{(n)}\}$, where ${\bf u}^{(1)},{\bf u}^{(2)},\ldots, {\bf u}^{(n)}\in \mathbb{Q}^m\backslash\{0\}$, is said to be orthogonal if $\langle {\bf u}^{(i)},{\bf u}^{(j)}\rangle=0$ for $i\neq j$, and orthonormal if it is orthogonal and $\langle {\bf u}^{(i)},{\bf u}^{(i)}\rangle=1$ for $i=1,2,\ldots,n$.
\end{Def}

It is easy to see that, a square  matrix $A\in \mathbb{Q}^{m\times m}$ is unitary if and only if its column (row) vectors form an orthonormal basis of $\mathbb{Q}^m$.

\begin{Prop} \cite{WLZZ18}
Let ${\bf u}^{(1)}\in \mathbb{Q}^m$ be a unit vector. Then there exist unit vectors ${\bf u}^{(2)},\ldots,{\bf u}^{(m)}\in \mathbb{Q}^m$ such that $\{{\bf u}^{(1)},{\bf u}^{(2)},\ldots,{\bf u}^{(m)}\}$ is an orthogonal set. 
\end{Prop}

Every $A\in \mathbb{Q}^{m\times m}$ can be written as $A=\hat A_1+\hat A_2\jj$, where $\hat A_1$ and $\hat A_2$ are $m\times m$ complex matrices. Using this expression of quaternion matrices, we can obtain many properties of the quaternion matrices, such as every quaternion matrix has a Jordan form under similar transformation \cite{ZW01}.










\subsection{Dual numbers}\label{Du-number}
Denote by $\mathbb D$ the set of the dual numbers. A dual number $q\in \mathbb{D}$ has the form $q = q_{\rm st} + q_{\rm I}\epsilon$, where $q_{\rm st}$ and $q_{\rm I}$ are real numbers,  and $\epsilon$ is the infinitesimal unit, satisfying $\epsilon^2 = 0$ but $\epsilon\neq 0$.   We call $q_{\rm st}$ the real part or the standard part of $q$, and $q_{\rm I}$ the dual part or the infinitesimal part of $q$.  The infinitesimal unit $\epsilon$ is commutative in multiplication with 
quaternion numbers. 
 If $q_{\rm st} \not = 0$, we say that $q$ is appreciable, otherwise, we say that $q$ is infinitesimal. The dual numbers form a commutative algebra of dimension two over the reals.

In \cite{QLY21}, a total order was introduced for dual numbers. Suppose $p = p_{\rm st} + p_{\rm I}\epsilon, q = q_{\rm st} + q_{\rm I}\epsilon \in \mathbb D$.  We have $q < p$ if $q_{\rm st} < p_{\rm st}$, or $q_{\rm st} = p_{\rm st}$ and $q_{\rm I} < p_{\rm I}$.  We have $q = p$ if $q_{\rm st} = p_{\rm st}$ and $q_{\rm I} = p_{\rm I}$.   Thus, if $q > 0$, we say that $q$ is a positive dual number; and if $q \ge 0$, we say that $q$ is a nonnegative dual number.  Denote the set of nonnegative dual numbers by $\mathbb D_+$, and the set of positive dual numbers by $\mathbb D_{++}$. For given $p = p_{\rm st} + p_{\rm I}\epsilon, q = q_{\rm st} + q_{\rm I}\epsilon \in \mathbb D$, and an integer $k$,
we have
\begin{equation} \label{e1}
p + q =p_{\rm st}+q_{\rm st} +(p_{\rm I}+q_{\rm I})\epsilon,~~~~pq = p_{\rm st}q_{\rm st} +(p_{\rm st}q_{\rm I}+p_{\rm I} q_{\rm st})\epsilon\end{equation}
and
\begin{equation} \label{e2}
q^k = q_{\rm st}^k + kq_{\rm st}^{k-1}q_{\rm I} \epsilon.
\end{equation}
The absolute value \cite{QLY21} of $q \in \mathbb D$ is defined by
\begin{equation} \label{e5}
|q| = \left\{ \begin{array}{ll}|q_{\rm st}| + {\rm sgn}(q_{\rm st})q_{\rm I}\epsilon, & {\rm if~}  q_{\rm st} \not = 0, \\
|q_{\rm I}|\epsilon, & {\rm otherwise},  \end{array}  \right.
\end{equation}
where for any $u \in \mathbb R$,
$${\rm sgn}(u) =  \left\{ \begin{aligned} 1, & \ {\rm if}\  u > 0, \\ 0, &   \ {\rm if}\  u = 0, \\
-1, &   \ {\rm if}\  u < 0.  \end{aligned}  \right. $$
For given $q=q_{\rm st}+q_{\rm I}\epsilon$, if $q$ is appreciable, then $q$ is nonsingular and $q^{-1} = q_{\rm st}^{-1} - q_{\rm st}^{-1}q_{\rm I} q_{\rm st}^{-1}\epsilon$. If $q$ is infinitesimal, then $q$ is not nonsingular. If $q$ is nonnegative and appreciable, then the square root of $q$ is still a nonnegative dual number. If $q$ is positive and appreciable, we have
\begin{equation} \label{e4}
\sqrt{q} = \sqrt{q_{\rm st}} + {q_{\rm I} \over 2\sqrt{q_{\rm st}}}\epsilon.
\end{equation}
When $q=0$, we have $\sqrt{q} = 0$.

\begin{Prop}\label{P6.5}
Let $p,q\in \mathbb{D}$. Then, we have the following conclusions.

(a) If $p, q \in \mathbb D_{+}$, then $pq \in \mathbb D_{+}$.

(b) If $p, q \in \mathbb D_{++}$ and at least one of them is appreciable, then $pq \in \mathbb D_{++}$.

(c) $|p|=p$ if $p\geq 0$, $|p|>p$ otherwise.

(d) $|p|=\sqrt{p^2}$ if $p$ is appreciable.

(e) If $p, q \in \mathbb D_{++}$ and are both appreciable, then $\sqrt{pq}=\sqrt{p}\sqrt{q}$.

(f) If $q\in \mathbb D_{++}$ and is appreciable, then $p-q \in \mathbb D_{+}$ implies $\sqrt{p}-\sqrt{q} \in \mathbb D_{+}$.

\begin{proof}
The proofs of (a)-(d) can be found in \cite{QLY21}. Since $p, q \in \mathbb D_{++}$ are both appreciable, by (\ref{e4}), we can verify (e) holds. Now we prove (f). Let $p=p_{\rm st}+p_{\rm I}\epsilon$ and $q=q_{\rm st}+q_{\rm I}\epsilon$. If $0<q_{\rm st}<p_{\rm st}$, the conclusion is clear. If $0<q_{\rm st}=p_{\rm st}$, then $q_{\rm I}\leq p_{\rm I}$ since $p-q \in \mathbb D_{+}$, which implies $\sqrt{p}-\sqrt{q} \in \mathbb D_{+}$ by (\ref{e4}).
\end{proof}

\end{Prop}


\subsection{Dual quaternions and dual quaternion matrices}

A dual quaternion $q$ has the form $q = q_{\rm st} + q_{\rm I}\epsilon$,
where $q_{\rm st}, q_{\rm I} \in \mathbb {Q}$ are the standard part and the infinitesimal part of $q$, respectively. Denote by $\mathbb {DQ}$ the set of dual quaternions. The conjugate of $q$ is $\bar q = \bar q_{\rm st} + \bar q_{\rm I}\epsilon$. See \cite{BK20, CKJC16, Ke12}. Similar to dual numbers, if $q_{\rm st} \not = 0$, then we say that $q$ is appreciable, otherwise, we say that $q$ is infinitesimal. We can derive that $q$ is invertible if and only if  $q$ is appreciable. In this case, we have $q^{-1} = q_{\rm st}^{-1} - q_{\rm st}^{-1}q_{\rm I} q_{\rm st}^{-1} \epsilon$. The magnitude of $q\in \mathbb{DQ}$ is defined as
\begin{equation} \label{e7}
\displaystyle|q| := \left\{ \begin{array}{ll} |q_{\rm st}| +\displaystyle {(q_{\rm st}\bar q_{\rm I}+q_{\rm I} \bar q_{\rm st}) \over 2|q_{\rm st}|}\epsilon, & \ {\rm if}\  q_{\rm st} \not = 0, \\
|q_{\rm I}|\epsilon, &  \ {\rm otherwise},
\end{array} \right.
\end{equation}
which is a dual number.    Note that this definition reduces to the definition of the absolute function if $q \in \mathbb D$, and the definition of the magnitude of a quaternion if $q \in \mathbb Q$.

Denote by $\mathbb{DQ}^{m\times n}$ the set of $m\times n$ dual quaternion matrices. Then $A\in \mathbb{DQ}^{m\times n}$ can
be written as $A = A_{\rm st} + A_{\rm I}\epsilon$, where $A_{\rm st}, A_{\rm I}\in \mathbb{Q}^{m\times n}$ are the standard part and the infinitesimal part of $A$, respectively.  If $A_{\rm st} \not = O$, we say that $A$ is appreciable, otherwise, we say that $A$ is infinitesimal. Note that for a dual quaternion matrix $A = A_{\rm st} + A_{\rm I}\epsilon$, if $A_{\rm st}=O$, the analysis for $A$ will be analogous to that for the quaternion matrix $A_{\rm I}$. Thus, unless otherwise stated, we will assume that the dual quaternion matrix $A$ is appreciable throughout the paper. It is obvious that when $n=1$, dual quaternion matrix $A$ reduces to the dual quaternion column vector with $m$ components. In this case, $\mathbb{DQ}^{m\times 1}$ is abbreviated as $\mathbb{DQ}^m$. For given ${\bf u}=(u_1,u_2,\ldots,u_m)^\top$ and ${\bf v}=(v_1,v_2,\ldots,v_m)^\top$ in $\mathbb{DQ}^m$, denote by $\langle {\bf u},{\bf v}\rangle$ the dual quaternion-valued inner product, i.e., $\langle {\bf u},{\bf v}\rangle=\sum_{i=1}^m\bar v_iu_i$. Accordingly, the orthogonality and orthonormality of vectors in $\mathbb{DQ}^m$ can be defined similarly to vectors in $\mathbb{Q}^m$. Denote by $[{\bf 0}]$ the set of all infinitesimal vectors in $\mathbb{DQ}^m$.  

For given $A\in \mathbb{DQ}^{m\times n}$, the transpose of $A$ is denoted as $A^\top = (a_{ji})$, the conjugate of $A$ is denoted as ${\bar A} = (\bar a_{ij})$, and the conjugate transpose of $A$ is denoted as $A^*=(\bar a_{ji})=\bar A^\top$. It is obvious that $A^\top = A_{\rm st}^\top + A_{\rm I}^\top\epsilon$, $\bar A = \bar A_{\rm st}+\bar A_{\rm I}\epsilon$ and $A^* = A_{\rm st}^* + A_{\rm I}^*\epsilon$. A square matrix $A\in \mathbb{DQ}^{m\times m}$ is called nonsingular (invertible) if $AB = BA = I_m$ for some $B\in \mathbb{DQ}^{m\times m}$. In that case, we denote $A^{-1} = B$. A square matrix $A\in \mathbb{DQ}^{m\times m}$ is called normal if $AA^*=A^*A$, Hermitian if $A^*=A$, and unitary if $A$ is nonsingular and $A^{-1}=A^*$. We have $(AB)^{-1}=B^{-1}A^{-1}$ if $A$ and $B$ are nonsingular, and $(A^*)^{-1}=(A^{-1})^*$ if $A$ is nonsingular. It is easy to see that a square matrix $U=[{\bf u}_1,{\bf u}_2,\ldots,{\bf u}_m]\in \mathbb{DQ}^{m\times m}$ is unitary if and only if $\{{\bf u}_1,{\bf u}_2,\ldots,{\bf u}_m\}$ form an orthonormal basis of $\mathbb{DQ}^m$, i.e., it is orthonormal and any vector ${\bf x}$ in $\mathbb {DQ}^m$ can be written as ${\bf x}=\sum_{i=1}^m{\bf u}_i\alpha_i$ for some $\alpha_1,\alpha_2,\ldots,\alpha_m\in \mathbb {DQ}$.


Suppose that $A\in \mathbb{DQ}^{m\times m}$. If there are $\lambda\in \mathbb{DQ}$, and ${\bf x}\in \mathbb{DQ}^m$, where ${\bf x}$ is appreciable,  such that
\begin{equation}\label{RightEig}
A{\bf x} = {\bf x}\lambda,
\end{equation}
then we say that $\lambda$ is a right eigenvalue of $A$, with ${\bf x}$ as an associated right eigenvector. If $A\in \mathbb{DQ}^{m\times m}$ is Hermitian, then $\lambda$ in (\ref{RightEig}) is a dual number \cite{QL21}, hence (\ref{RightEig}) implies $A{\bf x} = \lambda{\bf x}$, i.e., $\lambda$ is also a left eigenvalue of $A$, with ${\bf x}$ as an associated left eigenvector. In this case, $\lambda$ is simply called an eigenvalue of $A$, and ${\bf x}$ an associated eigenvector. In particular, it was shown in \cite{QL21} that an $m\times m$ dual quaternion Hermitian matrix has exactly $m$ dual number eigenvalues. 



For any $A=(a_{ij})\in \mathbb{DQ}^{m\times m}$, the norm of $A$, which is a dual number, is defined by
\begin{equation}\label{Norm-DQV}
\|A\|=\left\{
\begin{array}{ll}
\displaystyle\sqrt{\sum_{i=1}^m\sum_{j=1}^n|a_{ij}|^2},&{\rm if~}A~{\rm is ~appreciable},\\
\displaystyle\sqrt{\sum_{i=1}^m\sum_{j=1}^n|(a_{ij})_{\rm I}|^2} \epsilon,&{\rm otherwise}.
\end{array}
\right.
\end{equation}
\begin{Prop}\label{ReDQqin}
For any $q\in \mathbb{DQ}$, it holds that $|{\rm Re}(q)|\leq |q|$.
\end{Prop}
\begin{proof}
Without loss of generality, we assume $q\neq 0$. Write $q$ as $q=q_{\rm st}+q_{\rm I}\epsilon$ with $q_{\rm st}=(q_{\rm st})_0+{\rm Im}(q_{\rm st})$ and $q_{\rm I}=(q_{\rm I})_0+{\rm Im}(q_{\rm I})$. Then ${\rm Re}(q)=(q_{\rm st})_0+(q_{\rm I})_0\epsilon$. There are three cases: (a) $q_{\rm st}=0$, (b) $(q_{\rm st})_0\neq0$, and (c) $(q_{\rm st})_0=0$ but $q_{\rm st}\neq0$. In case (a), we have $(q_{\rm st})_0=0$. Consequently, since $|(q_{\rm I})_0|\leq |q_{\rm I}|$, it holds that $|{\rm Re}(q)|=|(q_{\rm I})_0|\varepsilon\leq |q_{\rm I}|\epsilon=|q|$, which means that the conclusion  is true. In case (b), we have $q_{\rm st}\neq 0$. If ${\rm Im}(q_{\rm st})\neq0$, then $|(q_{\rm st})_0|<|q_{\rm st}|$, which implies $|{\rm Re}(q)|<|q|$ by (\ref{e5}) and (\ref{e7}). If ${\rm Im}(q_{\rm st})=0$, then $|q_{\rm st}|=|(q_{\rm st})_0|$, hence, by (\ref{e7}), we know
$$
|q|=|(q_{\rm st})_0|+\frac{(q_{\rm st})_0(q_{\rm I})_0}{|(q_{\rm st})_0|}\epsilon=|(q_{\rm st})_0|+{\rm sgn}((q_{\rm st})_0)(q_{\rm I})_0\epsilon=|{\rm Re}(q)|.
$$
Therefore, we obtain the desired conclusion. In case (c), since $q_{\rm st}\neq0$ and $(q_{\rm st})_0=0$, we know $|q_{\rm st}|>0$ and $|{\rm Re}(q)|=|(q_{\rm I})_0|\epsilon$. It is obvious that $|{\rm Re}(q)|=|(q_{\rm I})_0|\epsilon<|q|$, which means that the conclusion is true.
\end{proof}

\begin{Prop} \cite{QLY21}\label{p6.3}
For any $\vx =\vx_{\rm st} + \vx_{\rm I}\epsilon\in{\mathbb {DQ}}^m$ with $\vx_{\rm st}\neq \0$, it holds that
\begin{equation}\label{e12-1}
\|\vx\|=\|\vx_{\rm st}\|+\displaystyle\frac{\langle\vx_{\rm st},\vx_{\rm I}\rangle+\langle\vx_{\rm I},\vx_{\rm st}\rangle}{2\|\vx_{\rm st}\|}\epsilon.
\end{equation}
\end{Prop}


\section{Right linear independency of dual quaternion vectors and  nonsingularity of dual quaternion matrices}\label{RighIN-Rank}
We first introduce the following three concepts of right linear independence of dual quaternion vectors, all of which are generalizations of the right linear independence of quaternion vectors.
\begin{Def}\label{Def-DQVecIn}
Let ${\bf u}^{(1)},{\bf u}^{(2)},\ldots,{\bf u}^{(n)}\in \mathbb{DQ}^m$. We say that $\{{\bf u}^{(1)},{\bf u}^{(2)},\ldots,{\bf u}^{(n)}\}$ is right linearly independent, if for any $\alpha_1,\alpha_2,\ldots,\alpha_n\in \mathbb{DQ}$,
$$
{\bf u}^{(1)}\alpha_1+{\bf u}^{(2)}\alpha_2+\ldots+{\bf u}^{(n)}\alpha_n={\bf 0}~~~\Rightarrow~~~\alpha_1=\alpha_2=\ldots=\alpha_n=0.
$$
\end{Def}

\begin{Def}\label{Def-DQVecAppIn}
Let ${\bf u}^{(1)},{\bf u}^{(2)},\ldots,{\bf u}^{(n)}\in \mathbb{DQ}^m$. We say that $\{{\bf u}^{(1)},{\bf u}^{(2)},\ldots,{\bf u}^{(n)}\}$ is right appreciably linearly independent, if for any $\alpha_1,\alpha_2,\ldots,\alpha_n\in \mathbb{DQ}$, $${\bf u}^{(1)}\alpha_1+{\bf u}^{(2)}\alpha_2+\ldots+{\bf u}^{(n)}\alpha_n {~is~ infinitesimal}~~\Rightarrow~~\alpha_1,\alpha_2,\ldots,\alpha_n ~{ are~ all~  infinitesimal}.$$
\end{Def}

\begin{Def}\label{Def-DQVecWeakIn}
Let ${\bf u}^{(1)},{\bf u}^{(2)},\ldots,{\bf u}^{(n)}\in \mathbb{DQ}^m$. We say that $\{{\bf u}^{(1)},{\bf u}^{(2)},\ldots,{\bf u}^{(n)}\}$ is weakly right linearly independent, if for any $\alpha_1,\alpha_2,\ldots,\alpha_n\in \mathbb{DQ}$,
$$
{\bf u}^{(1)}\alpha_1+{\bf u}^{(2)}\alpha_2+\ldots+{\bf u}^{(n)}\alpha_n={\bf 0}~~~\Rightarrow~~~\alpha_1,\alpha_2,\ldots,\alpha_n~{  are ~all ~infinitesimal}.
$$
\end{Def}

From Definition \ref{Def-DQVecIn}, it is easy to see that the right linear independence of $\{{\bf u}^{(1)},{\bf u}^{(2)},\ldots,{\bf u}^{(n)}\}$ is essentially that  $A{\bf x}={\bf 0}$ has only a unique zero solution in $\mathbb{DQ}^n$,  where $A=[{\bf u}^{(1)},{\bf u}^{(2)},\ldots,{\bf u}^{(n)}]\in \mathbb{DQ}^{m\times n}$. From Definition \ref{Def-DQVecWeakIn}, we know that the weak right linear independence of $\{{\bf u}^{(1)},{\bf u}^{(2)},\ldots,{\bf u}^{(n)}\}$ is essentially that  $A{\bf x}={\bf 0}$ has only infinitesimal solutions in $\mathbb{DQ}^n$.  Moreover, from Definitions \ref{Def-DQVecIn} and \ref{Def-DQVecWeakIn}, we know that, if $\{{\bf u}^{(1)},{\bf u}^{(2)},\ldots,{\bf u}^{(n)}\}$
is right linearly independent, then it must be weakly right linearly independent. The following example shows that, a dual quaternion vector set, which is weakly right linearly independent, is not necessarily right linearly independent.

\begin{example} Let ${\bf u}^{(1)}=(1,0)^\top \epsilon, {\bf u}^{(2)}=(0,1)^\top \epsilon\in \mathbb{DQ}^2$. It is easy to see that $\{{\bf u}^{(1)}, {\bf u}^{(2)}\}$ is weakly right linearly independent. However, it is obvious that $\{{\bf u}^{(1)}, {\bf u}^{(2)}\}$ is not right linearly independent, since ${\bf u}^{(1)}\alpha_1+{\bf u}^{(2)}\alpha_2={\bf 0}$ for $\alpha_1=\alpha_2=\epsilon\neq0$.
\end{example}

\begin{Prop}\label{PropAppIN1}
Let ${\bf u}^{(1)},{\bf u}^{(2)},\ldots,{\bf u}^{(n)}\in \mathbb{DQ}^m$ with ${\bf u}^{(i)}=({\bf u}^{(i)})_{\rm st}+({\bf u}^{(i)})_{\rm I}\epsilon$ for $i=1,2,\ldots,n$. Then $\{{\bf u}^{(1)},{\bf u}^{(2)},\ldots,{\bf u}^{(n)}\}$ is right appreciably linearly independent, if and only if $$\left\{({\bf u}^{(1)})_{\rm st},({\bf u}^{(2)})_{\rm st},\ldots,({\bf u}^{(n)})_{\rm st}\right\}$$ is right linearly independent in the sense of Definition \ref{Def-RIn}.
\end{Prop}

\begin{proof}
Let $\{{\bf u}^{(1)},{\bf u}^{(2)},\ldots,{\bf u}^{(n)}\}$ be right appreciably linearly independent. Suppose $$\sum_{i=1}^n({\bf u}^{(i)})_{\rm st}\alpha_i=0,$$ where $\alpha_i\in \mathbb{Q}$, which implies $$\sum_{i=1}^n{\bf u}^{(i)}\alpha_i=\left(\sum_{i=1}^n({\bf u}^{(i)})_{\rm I}\alpha_i\right)\epsilon,$$
i.e., $\sum_{i=1}^n{\bf u}^{(i)}\alpha_i$ is infinitesimal. By Definition \ref{Def-DQVecAppIn}, we know that $\alpha_1,\alpha_2,\ldots,\alpha_n$ are all infinitesimal, i.e., $\alpha_1=\alpha_2=\ldots=\alpha_n=0$. Hence, $\left\{({\bf u}^{(1)})_{\rm st},({\bf u}^{(2)})_{\rm st},\ldots,({\bf u}^{(n)})_{\rm st}\right\}$ is right linearly independent in the sense of Definition \ref{Def-RIn}.

Conversely, if $\left\{({\bf u}^{(1)})_{\rm st},({\bf u}^{(2)})_{\rm st},\ldots,({\bf u}^{(n)})_{\rm st}\right\}$ is right linearly independent in the sense of Definition \ref{Def-RIn}, then by Definition \ref{Def-DQVecAppIn}, we know immediately that $\{{\bf u}^{(1)},{\bf u}^{(2)},\ldots,{\bf u}^{(n)}\}$ is right appreciably linearly independent.
\end{proof}

\begin{Prop}\label{PropAppIN2}
Let ${\bf u}^{(1)},{\bf u}^{(2)},\ldots,{\bf u}^{(n)}\in \mathbb{DQ}^m$. If $\{{\bf u}^{(1)},{\bf u}^{(2)},\ldots,{\bf u}^{(n)}\}$ is right appreciably linearly independent, then $\{{\bf u}^{(1)},{\bf u}^{(2)},\ldots,{\bf u}^{(n)}\}$ is right linearly independent.
\end{Prop}

\begin{proof}
Let ${\bf u}^{(i)}=({\bf u}^{(i)})_{\rm st}+({\bf u}^{(i)})_{\rm I}\epsilon$, where $({\bf u}^{(i)})_{\rm st}, ({\bf u}^{(i)})_{\rm I}\in \mathbb{Q}^m$ for $i=1,2,\ldots,n$. For any $\alpha_1,\alpha_2,\ldots,\alpha_n\in \mathbb{DQ}$, it is obvious that
$$\sum_{i=1}^n{\bf u}^{(i)}\alpha_i=\sum_{i=1}^n({\bf u}^{(i)})_{\rm st}(\alpha_i)_{\rm st}+\left(\sum_{i=1}^n({\bf u}^{(i)})_{\rm st}(\alpha_i)_{\rm I}+\sum_{i=1}^n({\bf u}^{(i)})_{\rm I}(\alpha_i)_{\rm st}\right)\epsilon.$$
Suppose that $\{{\bf u}^{(1)},{\bf u}^{(2)},\ldots,{\bf u}^{(n)}\}$ is right appreciably linearly independent. Now we prove that $\{{\bf u}^{(1)},{\bf u}^{(2)},\ldots,{\bf u}^{(n)}\}$ must be right linearly independent. If $\sum_{i=1}^n{\bf u}^{(i)}\alpha_i={\bf 0}$, then we have
\begin{equation}\label{StImEq-1}
\left\{
\begin{array}{l}
\displaystyle\sum_{i=1}^n({\bf u}^{(i)})_{\rm st}(\alpha_i)_{\rm st}={\bf 0},\\
\displaystyle\sum_{i=1}^n({\bf u}^{(i)})_{\rm st}(\alpha_i)_{\rm I}+\sum_{i=1}^n({\bf u}^{(i)})_{\rm I}(\alpha_i)_{\rm st}={\bf 0}.
\end{array}
\right.
\end{equation}
By Proposition \ref{PropAppIN1}, the first equality in (\ref{StImEq-1}) and Definition \ref{Def-RIn}, we know that $(\alpha_i)_{\rm st}=0$ for $i=1,2,\ldots,n$. Consequently, from the second equality in (\ref{StImEq-1}), we have
$$
\sum_{i=1}^n\left({\bf u}^{(i)}\right)_{\rm st}(\alpha_i)_{\rm I}={\bf 0}.
$$
Since $\left\{({\bf u}^{(1)})_{\rm st},({\bf u}^{(2)})_{\rm st},\ldots,({\bf u}^{(n)})_{\rm st}\right\}$ is right linearly independent in the sense of Definition \ref{Def-RIn}, we know  that $(\alpha_i)_{\rm I}=0$ for $i=1,2,\ldots,n$. Hence, it holds that $\alpha_i=0$ for $i=1,2,\ldots,n$. This means that  $\left\{{\bf u}^{(1)},{\bf u}^{(2)},\ldots,{\bf u}^{(n)}\right\}$ is right linearly independent.
\end{proof}

\begin{Prop}\label{OnlyZeroS}
Let $A=A_{\rm st}+A_{\rm I}\epsilon\in \mathbb{DQ}^{m\times n}$. Then $A{\bf x}={\bf 0}$ has only a unique zero solution in $\mathbb{DQ}^n$, if and only if ${\rm rank}(A_{\rm st})=n$.
\end{Prop}

\begin{proof}
Since $A{\bf x}=A_{\rm st}{\bf x}_{\rm st}+(A_{\rm st}{\bf x}_{\rm I}+A_{\rm I}{\bf x}_{\rm st})\epsilon$, where ${\bf x}_{\rm st}, {\bf x}_{\rm I}\in \mathbb{Q}^n$, it is easy to see that $A{\bf x}={\bf 0}$ is equivalent to
\begin{equation}\label{MatEq1}
\left[\begin{array}{cc}
A_{\rm st}&O\\
A_{\rm I}&A_{\rm st}
\end{array}\right]\left[\begin{array}{c}{\bf x}_{\rm st}\\
{\bf x}_{\rm I}\end{array}\right]=\left[\begin{array}{c}
{\bf 0}\\
{\bf 0}
\end{array}\right].
\end{equation}
If ${\rm rank}(A_{\rm st})=n$, then $A_{\rm st}{\bf z}={\bf 0}$ has only a unique zero solution in $\mathbb{Q}^n$. Suppose $A{\bf x}={\bf 0}$. We know that $A_{\rm st}{{\bf x}}_{\rm st}={\bf 0}$ by (\ref{MatEq1}). Moreover, by quaternion matrix theory, we know ${{\bf x}}_{\rm st}={\bf 0}$ since ${\rm rank}(A_{\rm st})=n$. Consequently, by (\ref{MatEq1}), it holds that $A_{\rm st}{{\bf x}}_{\rm I}={\bf 0}$, which implies ${{\bf x}}_{\rm I}={\bf 0}$ due to the same reason. Hence ${{\bf x}}={\bf 0}$.

Conversely, if $A{\bf x}={\bf 0}$ has only a unique zero solution, i.e., (\ref{MatEq1}) has only a unique zero solution,  by quaternion matrix theory, we know that $${\rm rank}\left(\left[\begin{array}{cc}
A_{\rm st}&O\\
A_{\rm I}&A_{\rm st}
\end{array}\right]\right)=2n,$$ which implies ${\rm rank}(A_{\rm st})=n$. We  complete the proof.
\end{proof}

From Proposition \ref{OnlyZeroS}, we know that the right linear independence of $\Omega=\{{\bf u}^{(1)},{\bf u}^{(2)},\ldots,{\bf u}^{(n)}\}$ is completely determined by the right linear independence of the set consist of the standard part vectors of $\Omega$.

Based upon the right linear independence of the set of dual quaternion vectors stated above, we can define the rank of dual quaternion matrices. The (weak) rank of the dual quaternion matrix $A\in \mathbb{DQ}^{m\times n}$ is defined to be the maximum number of columns of $A$ which are (weakly) right linearly independent, and denoted by $({\rm rank}_w(A))$ ${\rm rank}(A)$.

\begin{Prop} \label{DQEQmn}
Let $A\in \mathbb{DQ}^{m\times n}$. If $m<n$, then $A{\bf x}={\bf 0}$ has a nonzero solution in $\mathbb{DQ}^{n}$.
\end{Prop}

\begin{proof}
It follows from the fact that ${\rm rank}(A)<n$.
\end{proof}

Now we discuss nonsingularity of dual quaternion matrices and some related basic properties. First, from the definition of nonsingularity of dual quaternion matrices, we know that, for given $A\in \mathbb{DQ}^{m\times m}$ and ${\bf b}\in \mathbb{DQ}^m$, if $A$ is nonsingular, then $A{\bf x}={\bf b}$ has a unique solution ${\bf x}=A^{-1}{\bf b}$. The following proposition shows that the nonsingular condition of a square dual quaternion  matrix  is similar to that of the quaternion matrix presented in \cite{Zh97}.


\begin{Prop}
Let $A,B\in \mathbb{DQ}^{m\times m}$. If $AB=I$, then $BA=I$.
\end{Prop}
\begin{proof}
Let $A=A_{\rm st}+A_{\rm I}\epsilon$ and $B=B_{\rm st}+B_{\rm I}\epsilon$. Since $AB=A_{\rm st}B_{\rm st}+(A_{\rm st}B_{\rm I}+A_{\rm I}B_{\rm st})\epsilon$, we know that $AB=I$ if and only if
\begin{equation}\label{DQEquat1}
\left\{
\begin{array}{l}
A_{\rm st}B_{\rm st}=I\\
A_{\rm st}B_{\rm I}+A_{\rm I}B_{\rm st}=O.
\end{array}
\right.
\end{equation}
Notice that the proposition is true for quaternion matrices \cite{Zh97}. From the first equality in (\ref{DQEquat1}), we know $B_{\rm st}A_{\rm st}=I$. Consequently, from the second equality in (\ref{DQEquat1}), we have $B_{\rm st}A_{\rm I}+B_{\rm I}A_{\rm st}=0$. Hence, $BA=B_{\rm st}A_{\rm st}+(B_{\rm st}A_{\rm I}+B_{\rm I}A_{\rm st})\epsilon=I$.
\end{proof}

\begin{Prop}\label{NonSinA}
Let $A=A_{\rm st}+A_{\rm I}\epsilon\in \mathbb{DQ}^{m\times m}$. Then $A$ is nonsingular if and only if $A_{\rm st}$ is nonsingular. More precisely, $A^{-1}=A_{\rm st}^{-1}-A_{\rm st}^{-1}A_{\rm I}A_{\rm st}^{-1}\epsilon$, provided $A_{\rm st}$ is nonsingular.
\end{Prop}

\begin{proof}
If $A=A_{\rm st}+A_{\rm I}\epsilon$ is nonsingular, there exists $B=B_{\rm st}+B_{\rm I}\epsilon\in \mathbb{DQ}^{m\times m}$ such that $AB=BA=I$, that is, $A_{\rm st}B_{\rm st}+(A_{\rm I}B_{\rm st}+A_{\rm st}B_{\rm I})\epsilon=B_{\rm st}A_{\rm st}+(B_{\rm I}A_{\rm st}+B_{\rm st}A_{\rm I})\epsilon=I$, which implies $A_{\rm st}B_{\rm st}=B_{\rm st}A_{\rm st}=I$. Hence $A_{\rm st}$ is nonsingular. Conversely, if $A_{\rm st}$ is nonsingular, then it is easy to verify that $AB=BA=I$, where $B=A_{\rm st}^{-1}-A_{\rm st}A_{\rm I}A_{\rm st}^{-1}\epsilon$. This means that $A$ is nonsingular.
\end{proof}

For $A=A_{\rm st}+A_{\rm I}\epsilon\in \mathbb{DQ}^{m\times n}$, we call the $2m\times 2n$ quaternion matrix
\begin{equation}\label{ADMatrix}
\chi_A:=\left[\begin{array}{cc}
A_{\rm st}&O\\
A_{\rm I}&A_{\rm st}\\
\end{array}
\right]\in \mathbb{Q}^{2m\times 2n},
\end{equation}
uniquely determined by $A$, the quaternion adjoint matrix or adjoint of the dual quaternion matrix $A$. 
It is obvious that $\chi_{I_m}=I_{2m}$. 
Furthermore, we have the following proposition.

\begin{Prop}\label{AdProp-1}
Let $A,B\in \mathbb{DQ}^{m\times m}$. Then we have

(a) $\chi_{AB}=\chi_A\chi_B$;

(b) $\chi_{A+B}=\chi_A +\chi_B$;

(c) $\chi_{A^{-1}}=(\chi_A)^{-1}$ if $A^{-1}$ exists.
\end{Prop}

\begin{proof}
The proofs of (a) and (b) are trivial. We now prove (c). First, by Proposition \ref{NonSinA},  it follows that $A_{\rm st}$ is nonsingular. Furthermore, by (\ref{ADMatrix}), we know
 $$
(\chi_{A})^{-1}=\left[\begin{array}{cc}
A^{-1}_{\rm st}&O\\
-A_{\rm st}^{-1}A_{\rm I}A_{\rm st}^{-1}&A^{-1}_{\rm st}\\
\end{array}
\right].
$$
On the other hand, by Proposition \ref{NonSinA}, we know $A^{-1}=A_{\rm st}^{-1}-A_{\rm st}^{-1}A_{\rm I}A_{\rm st}^{-1}\epsilon$, which means
$$
\chi_{A^{-1}}=\left[\begin{array}{cc}
A^{-1}_{\rm st}&O\\
-A_{\rm st}^{-1}A_{\rm I}A_{\rm st}^{-1}&A^{-1}_{\rm st}\\
\end{array}
\right].
$$
Hence, we obtain (c) and complete the proof.
\end{proof}

Using the quaternion adjoint matrix (\ref{ADMatrix}) of the dual quaternion matrix and the complex adjoint matrix \cite{ZW01} of the quaternion matrix , we can establish the connection between dual quaternion matrices and complex matrices.

\begin{Thm}\label{DQEigThm}
Let $A\in \mathbb{DQ}^{m\times m}$. Then the following statement are equivalent:

(a) $A$ is nonsingular;

(b) $A{\bf x}={\bf 0}$ has a unique solution ${\bf 0}$ in $\mathbb{DQ}^m$, i.e., $N(A):=\{{\bf x}\in \mathbb{DQ}^m~|~A{\bf x}={\bf 0}\}=\{{\bf 0}\}$.

(c) $\chi_A$  is nonsingular.

\end{Thm}

\begin{proof}
(a) $\Rightarrow$ (b): This is straight forward.

(b) $\Rightarrow$ (c): It is obvious that $A{\bf x}={\bf 0}$ is equivalent to
\begin{equation}\label{Q-NonSin-1}
\chi_A\left[\begin{array}{c}
{\bf x}_{\rm st}\\
{\bf x}_{\rm I}
\end{array}\right]={\bf 0},
\end{equation}
which implies that (b) is equivalent to that (\ref{Q-NonSin-1}) has only zero solution in $\mathbb{Q}^{2m}$. Consequently, by Theorem 4.3 in \cite{Zh97}, we know that $\chi_A$ is nonsingular.

(c) $\Rightarrow$ (a): Since $\chi_A$ is nonsingular, from the structure of $\chi_A$, must there exist $B\in \mathbb{Q}^{m\times m}$ such that $A_{\rm st}B=I_{m}$, which implies $A_{\rm st}$ is nonsingular. By Proposition \ref{NonSinA}, we know that (a) is true.
\end{proof}

\begin{Prop}\label{Rank-Nonsingular}
Let $A=A_{\rm st}+A_{\rm I}\epsilon\in \mathbb{DQ}^{m\times m}$. Then $A$ is nonsingular if and only if ${\rm rank}(A)=m$.
\end{Prop}

\begin{proof}
Suppose that $A$ is nonsingular, it is clear that $A{\bf x}={\bf 0}$ has only zero solution in $\mathbb{DQ}^m$, i.e., the set of column vectors of $A$ is right linearly independent, which means ${\rm rank(A)=m}$. Conversely, suppose ${\rm rank(A)=m}$. Then $A{\bf x}={\bf 0}$ has only zero solution. By Proposition \ref{OnlyZeroS}, we know ${\rm rank}(A_{\rm st})=m$. Consequently, by the well-known property in quaternion matrices, we know that $A_{\rm st}$ is nonsingular. Finally, by Proposition \ref{NonSinA}, we know that $A$ is nonsingular.
\end{proof}

Let $A\in \mathbb{DQ}^{m\times n}$, and denote $B=A^*A$. We have the following proposition.
\begin{Prop}
If ${\rm rank}(A)=n$, then ${\bf x}^*B{\bf x}\in \mathbb{D}_{++}$ is appreciable for any ${\bf x}\in \mathbb{DQ}^n$ with being appreciable.
\end{Prop}

\begin{proof}
First, since ${\rm rank}(A)=n$, it holds that ${\rm rank}(A_{\rm st})=n$ by Proposition \ref{OnlyZeroS}. For any ${\bf x}\in \mathbb{DQ}^n$ with being appreciable, denote ${\bf y}={\bf y}_{\rm st}+{\bf y}_{\rm I}\epsilon:=A{\bf x}$. It is clear that ${\bf x}^*B{\bf x}=(A{\bf x})^*A{\bf x}=\|{\bf y}_{\rm st}\|^2+(\langle{\bf y}_{\rm st},{\bf y}_{\rm I}\rangle+\langle{\bf y}_{\rm I},{\bf y}_{\rm st}\rangle)\epsilon\in \mathbb{D}_+$. We claim that $\|{\bf y}_{\rm st}\|^2>0$. In fact, if $\|{\bf y}_{\rm st}\|^2=0$, then ${\bf y}_{\rm st}={\bf 0}$, which implies $A{\bf x}={\bf y}_{\rm I}\epsilon$. Consequently, we know
$$
\left[\begin{array}{cc}
A_{\rm st}&O\\
A_{\rm I}&A_{\rm st}
\end{array}
\right]\left[\begin{array}{c}
{\bf x}_{\rm st}\\
{\bf x}_{\rm I}
\end{array}
\right]=\left[\begin{array}{c}
{\bf 0}\\
{\bf y}_{\rm I}
\end{array}\right].
$$
Hence $A_{\rm st}{\bf x}_{\rm st}={\bf 0}$, where ${\bf x}_{\rm st}\neq{\bf 0}$, which implies $A_{\rm st}{\bf z}={\bf 0}$ has a nonzero solution, and hence  ${\rm rank}(A_{\rm st})<n$. It is a contradiction.
Since $\|{\bf y}_{\rm st}\|^2>0$, we know that ${\bf x}^*B{\bf x}\in \mathbb{D}_{++}$.
\end{proof}

The following theorem characterizes the relationship between nonsingularity and eigenvalues of dual quaternion matrices.

\begin{Thm}\label{Eigen-StApp}
Let $A=A_{\rm st}+A_{\rm I}\epsilon\in \mathbb{DQ}^{m\times m}$. If $A$ is nonsingular, then every right eigenvalue of $A$ is appreciable.
\end{Thm}

\begin{proof}
Suppose that $A$ is nonsingular. By Proposition \ref{OnlyZeroS}, we know that $A_{\rm st}$ is nonsingular. If there is an infinitesimal right eigenvalue $\lambda=\lambda_{\rm I}\epsilon$ of $A$, with the associated right eigenvector ${\bf x}={\bf x}_{\rm st}+{\bf x}_{\rm I}\epsilon$, then by the definition of right eigenvalue of $A$, we have $A_{\rm st}{\bf x}_{\rm st}+(A_{\rm I}{\bf x}_{\rm st}+A_{\rm st}{\bf x}_{\rm I})\epsilon={\bf x}_{\rm st}\lambda_{\rm I}\epsilon$, which implies $A_{\rm st}{\bf x}_{\rm st}={\bf 0}$. This means that $A_{\rm st}$ is singular by Proposition \ref{SingularP}, since ${\bf x}_{\rm st}\neq {\bf 0}$. It is a contradiction.
\end{proof}

\begin{Prop}
Let $A\in \mathbb{DQ}^{m\times m}$. Suppose that $\lambda=\lambda_{\rm st}+\lambda_{\rm I}\epsilon\in \mathbb{DQ}$ is a right eigenvalue of $A$, with the associated right eigenvector ${\bf x}$. If $A$ is nonsingular, then $\lambda^{-1}$ is a right eigenvalue of $A^{-1}$, with the associated right eigenvector ${\bf x}$. \end{Prop}

\begin{proof}
Since $A$ is nonsingular, by Theorem \ref{Eigen-StApp}, $\lambda$ is appreciable, which implies that $\lambda^{-1}=\lambda_{\rm st}^{-1}-\lambda_{\rm st}^{-1}\lambda_{\rm I}\lambda_{\rm st}^{-1}\epsilon$, Consequently, from $A{\bf x}={\bf x}\lambda$ and $A^{-1}A=I$, it follows that ${\bf x}=A^{-1}{\bf x}\lambda$, which implies $A^{-1}{\bf x}={\bf x}\lambda^{-1}$ since $\lambda\lambda^{-1}=1$. Hence, we obtain the desired result and complete the proof.
\end{proof}


\section{Minimax principle for eigenvalues of dual quaternion Hermitian matrices}\label{MinimaxPrin}
In this section, we present a Courant-Fischer type minimax principle for eigenvalues of dual quaternion Hermitian matrices. Because $\mathbb{D}$ is a total order space in the meaning of total order stated in Section \ref{Du-number}, unless otherwise specified, for $p,q\in \mathbb{D}$, $p\leq q$, if and only if $q-p\in \mathbb{D}_+$, and the related maximization and minimization are also discussed in this sense. We start by recalling the following unitary decomposition theorem of dual quaternion Hermitian matrices.

\begin{Thm}\label{HUDec}\cite{QL21} Let $A=A_{\rm st}+A_{\rm I}\epsilon\in \mathbb{DQ}^{m\times m}$ be an Hermitian matrix. Then there are
unitary matrix $U\in \mathbb{DQ}^{m\times m}$ and a diagonal matrix $\Sigma\in \mathbb{DQ}^{m\times m}$ such that $A=U\Sigma U^*$, where
\begin{equation}\label{Sigmn}\Sigma := {\rm diag} (\lambda_1+\lambda_{1,1}\epsilon,\ldots,\lambda_1+\lambda_{1,k_1}\epsilon, \lambda_2+\lambda_{2,1}\epsilon,\ldots,\lambda_r+\lambda_{r,k_r}\epsilon),
\end{equation}
where $\lambda_1>\lambda_2>\ldots>\lambda_r$ are real numbers, $\lambda_i$ is a $k_i$-multiple right eigenvalue of $A_{\rm st}$, $\lambda_{i,1}\geq \lambda_{i,2}\geq\ldots\geq\lambda_{i,k_i}$ are also real numbers. Counting possible multiplicities $\lambda_{i,j}$, the form $\Sigma$ is unique.
\end{Thm}

\begin{Lem}\label{Lemma1}
Let $A\in \mathbb{DQ}^{m\times m}$ be an Hermitian matrix, and $\lambda^{(1)}\geq\lambda^{(2)}\geq\ldots\geq\lambda^{(m)}$ be right eigenvalues of $A$. Then we have
\begin{equation}\label{MaxMinLam-1}
\lambda^{(1)}=\max\left\{\|{\bf x}\|^{-2}({\bf x}^*A{\bf x})~|~{\bf x}\in \mathbb{DQ}^m\backslash[{\bf 0}]\right\}
\end{equation}
and
\begin{equation}\label{MaxMinLam-2}
\lambda^{(m)}=\min\left\{\|{\bf x}\|^{-2}({\bf x}^*A{\bf x})~|~{\bf x}\in \mathbb{DQ}^m\backslash[{\bf 0}]\right\}.
\end{equation}
\end{Lem}

\begin{proof}
Since $A$ is Hermitian, we know that ${\bf x}^*A{\bf x}\in \mathbb{D}$ for any ${\bf x}\in \mathbb{DQ}^m$. Moreover, by Theorem \ref{HUDec}, there exists a unitary matrix $U=[{\bf u}_1,{\bf u}_2,\ldots,{\bf u}_m]$ such that $A=U{\rm diag}(\lambda^{(1)},\lambda^{(2)},\ldots,\lambda^{(m)})U^*$. It is obvious that $\lambda^{(1)}=\lambda_1+\lambda_{1,1}\epsilon$ and $\lambda^{(m)}=\lambda_r+\lambda_{r,k_r}\epsilon$. For any ${\bf x}\in \mathbb{DQ}^m\backslash[{\bf 0}]$, it holds that
\begin{equation}\label{xAy}
{\bf x}^*A{\bf x}={\bf x}^*U{\rm diag}(\lambda^{(1)},\lambda^{(2)},\ldots,\lambda^{(m)})U^*{\bf x}=\sum_{i=1}^m\lambda^{(i)}\bar y_iy_i,
\end{equation}
where ${\bf y}:=(y_1,y_2,\ldots,y_m)^\top=U^*{\bf x}$. Since $U$ is unitary, it is obvious that
$$\|{\bf y}\|^2=\langle{\bf y},{\bf y}\rangle={\bf x}^*UU^*{\bf x}=\langle{\bf x},{\bf x}\rangle=\|{\bf x}\|^2.
$$ Since $\bar y_iy_i\geq 0$ and $\lambda^{(1)}\geq \lambda^{(i)}$ for $i=1,2,\ldots,m$, by (\ref{xAy}) and Proposition \ref{P6.5} (a), we have
$$
{\bf x}^*A{\bf x}=\sum_{i=1}^m\lambda^{(i)}\bar y_iy_i\leq\lambda^{(1)}\sum_{i=1}^m\bar y_iy_i=\lambda^{(1)}\|{\bf x}\|^2,
$$ which implies that $\|{\bf x}\|^{-2}({\bf x}^*A{\bf x})\leq\lambda^{(1)}$ for any ${\bf x}\in \mathbb{DQ}^m\backslash[{\bf 0}]$. On the other hand, by taking $\tilde {\bf x}={\bf u}_1$, we know that $\tilde{\bf x}^*A\bar {\bf x}=\lambda^{(1)}$ and $\|\tilde{\bf x}\|=1$.
Therefore, (\ref{MaxMinLam-1}) holds. The proof of (\ref{MaxMinLam-2}) is similar.
\end{proof}


\begin{Lem}\label{Lemma2}
Let $A=A_{\rm st}+A_{\rm I}\epsilon\in \mathbb{DQ}^{m\times m}$ be an Hermitian matrix, and $\lambda^{(1)}\geq\lambda^{(2)}\geq\ldots\geq\lambda^{(m)}$ be right eigenvalues of $A$. 
Then we have
\begin{equation}\label{Lambda_k}
\lambda^{(k)}=\max\left\{\|{\bf x}\|^{-2}({\bf x}^*A{\bf x})~|~{\bf x}\in \mathbb{DQ}^m\backslash[{\bf 0}],~\langle{\bf u}_j,{\bf x}\rangle=0,~ j=1,2,\ldots,k-1\right\}
\end{equation}
for $k=2,3,\ldots,m$, and
\begin{equation}\label{Lambda^{(m-k)}}\lambda^{(m-k)}=\min\left\{\|{\bf x}\|^{-2}({\bf x}^*A{\bf x})~|~{\bf x}\in \mathbb{DQ}^m\backslash[{\bf 0}],~\langle{\bf u}_j,{\bf x}\rangle=0,~ j=k,k+1,\ldots,m\right\}
\end{equation}
for $k=1,2,\ldots,m-1$, where ${\bf u}_j$ denotes the $j$-th column vector of the unitary matrix $U$ in Theorem \ref{HUDec}.
\end{Lem}

\begin{proof}
We only prove (\ref{Lambda_k}). The expression (\ref{Lambda^{(m-k)}}) can be proved similarly. For every $k=2,3,\ldots,m$, denote $\mathbb{F}_k=\{{\bf x}\in \mathbb{DQ}^m\backslash[{\bf 0}],~\langle{\bf u}_j,{\bf x}\rangle=0, ~j=1,2,\ldots,k-1\}$, $U_1=[{\bf u}_{1},\ldots,{\bf u}_{k-1}]\in \mathbb{DQ}^{m\times (k-1)}$ and $U_2=[{\bf u}_{k},\ldots,{\bf u}_m]\in \mathbb{DQ}^{m\times (m-k+1)}$. It is obvious that $U=[U_1,U_2]$, and $U_2^*U=[O_{(m-k+1)\times (k-1)},I_{m-k+1}]$ since $U_2^*U_1=O_{(m-k+1)\times (k-1)}$ and $U_2^*U_2=I_{m-k+1}$. Since $\{{\bf u}_{1},{\bf u}_{2},\ldots,{\bf u}_m\}$ is a basis of $\mathbb{DQ}^m$, it is easy to see that for any ${\bf x}\in \mathbb{F}_k$, there exists ${\bf y}=(y_k, y_{k+1},\ldots,y_m)^\top\in \mathbb{DQ}^{m-k+1}$, such that ${\bf x}=U_2{\bf y}$. Consequently, we have
$${\bf x}^*A{\bf x}={\bf y}^*U_2^*U{\rm diag}(\lambda^{(1)},\lambda^{(2)},\ldots,\lambda^{(m)})U^*U_2{\bf y}=\sum_{i=k}^m\lambda_i\bar y_iy_i\leq \lambda_k{\bf y}^*{\bf y},$$
which implies, together with the fact $\|{\bf x}\|^2=\langle{\bf x},{\bf x}\rangle={\bf y}^*U_2^*U_2{\bf y}=\langle{\bf y},{\bf y}\rangle=\|{\bf y}\|^2$ from $U_2^*U_2=I_{m-k+1}$, that $\|{\bf x}\|^{-2}({\bf x}^*A{\bf x})\leq \lambda_k$ for any ${\bf x}\in {\mathbb{F}}_k$. Moreover, by taking $\tilde{\bf x}={\bf u}_k\in \mathbb{F}_k$, we $\tilde{\bf x}^*A\tilde{\bf x}=\lambda_k$. Hence, we claim that (\ref{Lambda_k}) holds, and complete the proof.
\end{proof}

We now present one of the main results in this section, which characterizes the minimax principle for eigenvalues of dual quaternion  Hermitian matrices.

\begin{Thm}\label{Lemma3}
Let $A\in \mathbb{DQ}^{m\times m}$ be Hermitian, and $\lambda^{(1)}\geq\lambda^{(2)}\geq\ldots\geq\lambda^{(m)}$ be eigenvalues of $A$. Then for $k=2,3,\ldots,m$, we have
\begin{equation}\label{MLambda_k}
\lambda^{(k)}=\min_{B\in \mathbb{DQ}^{m\times (k-1)}}\max_{{\bf x}\in N(B^*)\backslash[{\bf 0}]}\|{\bf x}\|^{-2}({\bf x}^*A{\bf x}),
\end{equation}
and it attains $\lambda^{(k)}$ when $B=[{\bf u}_1,{\bf u}_2,\ldots,{\bf u}_{k-1}]$;
and for $k=1,2,\ldots,m-1$, we have
\begin{equation}\label{MLambda_{n-k}}\lambda^{(m-k)}=\max_{C\in \mathbb{DQ}^{m\times k}}\min_{{\bf x}\in N(C^*)\backslash[{\bf 0}]}\|{\bf x}\|^{-2}({\bf x}^*A{\bf x}),
\end{equation}
and it attains the $\lambda^{(m-k)}$ when $C=[{\bf u}_{m-k+1},{\bf u}_{m-k+2},\ldots,{\bf u}_m]$, where ${\bf u}_i$ is the $i$th column of unitary matrix $U$ in  Theorem \ref{HUDec}. Here, for given $W\in \mathbb{DQ}^{p\times q}$, $N(W):=\{{\bf z}\in \mathbb{DQ}^{q}~|~W{\bf z}={\bf 0}\}$. 
\end{Thm}

\begin{proof}
We only prove the first conclusion. The second one can be proved similarly. Since $A$ is Hermitian, there exists a unitary matrix $U$ such that  $A=U{\rm diag}(\lambda^{(1)},\lambda^{(2)},\ldots,\lambda^{(m)})U^*$ holds. 
For any $B\in \mathbb{DQ}^{m\times (k-1)}$, denote $D=U^*B$. It is clear that $D^*=B^*U$ and $B^*=D^*U^*$, which implies that ${\bf x}\in N(B^*)$ if and only if ${\bf y}:=U^*{\bf x}\in N(D^*)$. Consequently, since $\|{\bf x}\|=\|{\bf y}\|$, for $k=2,3,\ldots,m$, we have
$$
\begin{array}{l}
\displaystyle\max_{{\bf x}\in N(B^*)\backslash[{\bf 0}]}\|{\bf x}\|^{-2}({\bf x}^*A{\bf x})\\
=\displaystyle\max_{{\bf y}\in N(B^*)\backslash[{\bf 0}]}\|{\bf y}\|^{-2}({\bf y}^*{\rm diag}(\lambda^{(1)},\lambda^{(2)},\ldots,\lambda^{(m)}){\bf y})\\
\geq\displaystyle\max_{({\bf y}_1,{\bf 0})^\top\in N(B^*)\backslash[{\bf 0}],{\bf y}_1\in \mathbb{DQ}^k}\|({\bf y}_1,{\bf 0})^\top\|^{-2}(({\bf y}_1^*,{\bf 0}^\top){\rm diag}(\lambda^{(1)},\lambda^{(2)},\ldots,\lambda^{(m)})({\bf y}_1^\top,{\bf 0}^\top)^\top)\\
=\displaystyle\max_{({\bf y}_1,{\bf 0})^\top\in N(B^*)\backslash[{\bf 0}],{\bf y}_1\in \mathbb{DQ}^k}\|{\bf y}_1\|^{-2}({\bf y}_1^*{\rm diag}(\lambda^{(1)},\lambda^{(2)},\ldots,\lambda^{(k)}){\bf y}_1)\\
\geq\displaystyle\min_{({\bf y}_1,{\bf 0})^\top\in N(B^*)\backslash[{\bf 0}],{\bf y}_1\in \mathbb{DQ}^k}\|{\bf y}_1\|^{-2}({\bf y}_1^*{\rm diag}(\lambda^{(1)},\lambda^{(2)},\ldots,\lambda^{(k)}){\bf y}_1)\\
\geq\displaystyle\min_{{\bf y}_1\in \mathbb{DQ}^k\backslash[{\bf 0}]}\|{\bf y}_1\|^{-2}({\bf y}_1^*{\rm diag}(\lambda^{(1)},\lambda^{(2)},\ldots,\lambda^{(k)}){\bf y}_1)\\
=\lambda^{(k)},
\end{array}
$$
which implies
$$
\max_{{\bf x}\in N(B^*)\backslash[{\bf 0}]}\|{\bf x}\|^{-2}({\bf x}^*A{\bf x})\geq \lambda^{(k)}~~~~~\forall~B\in \mathbb{DQ}^{m\times (k-1)}.
$$
Moreover, when $B=[{\bf u}_1,{\bf u}_2,\ldots,{\bf u}_{k-1}]$, it is easy to verify that
$$
\max_{{\bf x}\in N(B^*)\backslash[{\bf 0}]}\|{\bf x}\|^{-2}({\bf x}^*A{\bf x})= \lambda^{(k)}.
$$
Therefore, we obtain the desired result and complete the proof.
\end{proof}

Theorem  \ref{Lemma3} considers eigenvalues (including appreciable and infinitesimal parts) of a dual quaternion matrix as a whole to establish their relationship with related optimization models.  However, in many applications, we need consider separately the appreciable parts of eigenvalues of a dual quaternion matrix. To this end, we first recall the following proposition.

\begin{Prop} \cite{QL21}\label{Prop1} Suppose that $\lambda=\lambda_{{\rm st}}+\lambda_{\rm I}\epsilon\in \mathbb{DQ}$ is a right eigenvalue of $A=A_{{\rm st}}+A_{\rm I}\epsilon\in \mathbb{DQ}^{m\times m}$, with associated right eigenvector ${\bf x}={\bf x}_{\rm st} + {\bf x}_{{\rm I}}\epsilon\in \mathbb{DQ}^m$. Then
\begin{equation}\label{Reigenva}
\lambda=\frac{{\bf x}^*A{\bf x}}{\|{\bf x}\|^2}~~~~{\rm and}~~~~\lambda_{\rm st}=\frac{{\bf x}_{\rm st}^*A_{\rm st}{\bf x}_{\rm st}}{\|{\bf x}_{\rm st}\|^2}.
\end{equation}
Moreover, if $A$ is Hermitian, then
we have
\begin{equation}\label{Eigvalue-Vet-DQM}
\lambda_{\rm I}=\frac{{\bf x}_{\rm st}^*A_{\rm I}{\bf x}_{\rm st}}{\|{\bf x}_{\rm st}\|^2}.
\end{equation}
A dual quaternion Hermitian matrix has exactly $m$ dual number eigenvalues and no other right eigenvalues.
\end{Prop}

Based upon this proposition, we have the following theorem.

\begin{Thm}
Let $A=A_{\rm st}+A_{\rm I}\epsilon\in \mathbb{DQ}^{m\times m}$. Suppose that $A$ is Hermitian, and $\lambda^{(1)}\geq\lambda^{(2)}\geq\ldots\geq\lambda^{(m)}$ are eigenvalues of $A$. Then it holds that
\begin{equation}\label{StLamk}
\lambda_{\rm st}^{(1)}=\max_{{\bf x}\in \mathbb{Q}^m\backslash\{{\bf 0}\}}\frac{{\bf x}^*A_{\rm st}{\bf x}}{\|{\bf x}\|^2}~~~{\rm and}~~~\lambda_{\rm st}^{(k)}=\min_{B\in \mathbb{Q}^{m\times (k-1)}}\max_{{\bf x}\in N(B^*)\backslash\{{\bf 0}\}}\frac{{\bf x}^*A_{\rm st}{\bf x}}{\|{\bf x}\|^2},~~k=2,3,\ldots,m,
\end{equation}
and
\begin{equation}\label{LamMaxMin}
\lambda_{\rm min}(A_{\rm I})\leq\lambda_{\rm I}^{(k)}\leq \lambda_{\rm max}(A_{\rm I}), ~~k=1,2,\ldots,m,
\end{equation}
where $\lambda_{\rm st}^{(k)}$ and $\lambda_{\rm I}^{(k)}$ are the standard part and the dual part of $\lambda^{(k)}$ respectively. Here, $\lambda_{\rm max}(\cdot)$ and $\lambda_{\rm min}(\cdot)$ denote the largest and smallest eigenvalues of matrices respectively.  
\end{Thm}

\begin{proof}
By the definition of eigenvalues of dual quaternion Hermitian matrices and Proposition \ref{Prop1}, we know that $\lambda^{(1)}_{{\rm st}}\geq \lambda^{(2)}_{{\rm st}}\geq\ldots\geq\lambda^{(m)}_{{\rm st}}$ are eigenvalues of the quaternion Hermitian matrix $A_{{\rm st}}$. Consequently, by a similar method used in the proof of Lemma \ref{Lemma1}, but $\mathbb{DQ}^m\backslash[{\bf 0}]$ is replaced by $\mathbb{Q}^m\backslash\{{\bf 0}\}$, we can obtain the first expression in (\ref{StLamk}). By a similar method used in the proof of Theorem \ref{Lemma3}, but $\mathbb{DQ}^{m\times (k-1)}$ is replaced by $\mathbb{Q}^{m\times (k-1)}$, we can obtain the second expression in (\ref{StLamk}). Let ${\bf x}^{(k)}={\bf x}_{\rm st}^{(k)} + {\bf x}_{{\rm I}}^{(k)}\epsilon$ be an eigenvector of $A$, associate with the eigenvalue $\lambda^{(k)}=\lambda^{(k)}_{\rm st}+\lambda^{(k)}_{\rm I}\epsilon$. Denote $\tilde{\bf x}^{(k)}={\bf x}_{\rm st}^{(k)}/\|{\bf x}_{\rm st}^{(k)}\|^2$. It obvious that $\tilde{\bf x}^{(k)}\in \mathbb{Q}^m$ with $\|\tilde{\bf x}^{(k)}\|=1$ and $\lambda_{\rm I}^{(k)}=(\tilde{\bf x}^{(k)})^*A_{\rm I}\tilde{\bf x}^{(k)}$ by Proposition \ref{Prop1}. Then expression (\ref{LamMaxMin}) follows from the quaternion matrix theory, since $A_{\rm I}$ is a quaternion Hermitian matrix.
\end{proof}

We now propose the following proposition, which is a dual quaternion version of Cauchy-Schwarz inequality on $\mathbb{R}^m$.
\begin{Prop}\label{Ch-SW-In} (Cauchy-Schwarz inequality on $\mathbb{DQ}^m$)
For any ${\bf u},{\bf v}\in \mathbb{DQ}^m$, it holds that $$\|{\bf u}\|\|{\bf v}\|-|\langle{\bf u},{\bf v}\rangle|\in \mathbb{D}_+,$$
i.e., $|\langle{\bf u},{\bf v}\rangle|\leq\|{\bf u}\|\|{\bf v}\|$.
\end{Prop}

\begin{proof}
Take any ${\bf u}={\bf u}_{\rm st}+{\bf u}_{\rm I}\epsilon, {\bf v}={\bf v}_{\rm st}+{\bf v}_{\rm I}\epsilon\in \mathbb{DQ}^m$. If ${\bf u}$ and ${\bf v}$ are both infinitesimal, i.e., ${\bf u}_{\rm st}={\bf v}_{\rm st}={\bf 0}$, then the conclusion is clear, since $\langle{\bf u},{\bf v}\rangle=0$ in this case.

We first consider the case where one of ${\bf u}$ and ${\bf v}$ is appreciable and another one is infinitesimal. Without loss of generality, we assume that ${\bf u}_{\rm st}\neq{\bf 0}$ and ${\bf v}_{\rm st}={\bf 0}$. In this case, $\langle{\bf u},{\bf v}\rangle=\langle{\bf u}_{\rm st},{\bf v}_{\rm I}\rangle\epsilon$, which implies $|\langle{\bf u},{\bf v}\rangle|=|\langle{\bf u}_{\rm st},{\bf v}_{\rm I}\rangle|\epsilon$ by (\ref{e7}). Since ${\bf u}_{\rm st},{\bf v}_{\rm I}\in \mathbb{Q}^m$, by Proposition \ref{QCauchy-Inequality}, we have $|\langle{\bf u}_{\rm st},{\bf v}_{\rm I}\rangle|\leq \|{\bf u}_{\rm st}\|\|{\bf v}_{\rm I}\|$, so $|\langle{\bf u},{\bf v}\rangle|\leq \|{\bf u}_{\rm st}\|\|{\bf v}_{\rm I}\|\epsilon$. Moreover, by Proposition \ref{p6.3}, we have
$$
\left\{
\begin{array}{l}
\|{\bf u}\|=\|{\bf u}_{\rm st}\|+\displaystyle\frac{\langle\vu_{\rm st},\vu_{\rm I}\rangle+\langle\vu_{\rm I},\vu_{\rm st}\rangle}{2\|\vu_{\rm st}\|_2}\epsilon,\\
\|{\bf v}\|=\|{\bf v}_{\rm I}\|\epsilon,
\end{array}
\right.
$$
which implies that $\|{\bf u}\|\|{\bf v}\|=\|{\bf u}_{\rm st}\|\|{\bf v}_{\rm I}\|\epsilon$. Hence the desired conclusion holds.

Now we consider the case that ${\bf u}$ and ${\bf v}$ are both appreciable, i.e., ${\bf u}_{\rm st}\neq{\bf 0}$ and ${\bf v}_{\rm st}\neq{\bf 0}$. There are two situations: (a) $\langle{\bf u}_{\rm st},{\bf v}_{\rm st}\rangle=0$ and (b) $\langle{\bf u}_{\rm st},{\bf v}_{\rm st}\rangle\neq0$. When (a) occurs, the conclusion is clear, since $|\langle{\bf u},{\bf v}\rangle|=|\langle{\bf u}_{\rm st},{\bf v}_{\rm I}\rangle+\langle{\bf u}_{\rm I},{\bf v}_{\rm st}\rangle|\epsilon$, $\|{\bf u}\|\|{\bf v}\|=\|{\bf u}_{\rm st}\|\|{\bf v}_{\rm st}\|+p\epsilon$ for some $p\in \mathbb{D}$ and $0<\|{\bf u}_{\rm st}\|\|{\bf v}_{\rm st}\|$. When (b) occurs, we carefully follow the classical steps of the proof, keeping the order of terms in multiplication. It is clear that $\langle{\bf u}-{\bf v}\lambda,{\bf u}-{\bf v}\lambda\rangle\in \mathbb{D}_+$ for  any $\lambda\in \mathbb{DQ}$. More specifically, we have
$$
\begin{array}{lll}
\mathbb{D}_+&\ni& \langle{\bf u}-{\bf v}\lambda,{\bf u}-{\bf v}\lambda\rangle\\
&=&\displaystyle\sum_{i=1}^m \bar u_iu_i-\bar\lambda\sum_{i=1}^m \bar v_iu_i-\left(\sum_{i=1}^m \bar u_iv_i\right)\lambda+\bar \lambda\left(\sum_{i=1}^m \bar v_iv_i\right)\lambda\\
&=&\displaystyle\|{\bf u}\|^2-\bar\lambda\sum_{i=1}^m \bar v_iu_i-\left(\sum_{i=1}^m \bar u_iv_i\right)\lambda+\|{\bf v}\|^2\bar\lambda\lambda,
\end{array}
$$
where the last equality comes from the fact that $\|{\bf v}\|^2\in \mathbb{D}$. By taking $\lambda=\|{\bf v}\|^{-2}\sum_{i=1}^m \bar v_iu_i$, we know
$$
\displaystyle\|{\bf u}\|^2-\left(\sum_{i=1}^m\bar u_iv_i\right)\left(\sum_{i=1}^m \bar v_iu_i\right)\|{\bf v}\|^{-2}\in \mathbb{D}_+,
$$
which implies
\begin{equation}\label{ChSWIn}
\|{\bf u}\|^2\|{\bf v}\|^2-\langle{\bf u},{\bf v}\rangle\langle{\bf v},{\bf u}\rangle=\|{\bf u}\|^2\|{\bf v}\|^2-\left(\sum_{i=1}^m \bar u_iv_i\right)\left(\sum_{i=1}^m\bar v_iu_i\right)\in \mathbb{D}_+,
\end{equation}
by Proposition \ref{P6.5} and $\|{\bf v}\|^2\in \mathbb{D}_{++}$. Since $\langle{\bf v},{\bf u}\rangle=\overline{\langle{\bf u},{\bf v}\rangle}$, we know that $\langle{\bf u},{\bf v}\rangle\langle{\bf v},{\bf u}\rangle=|\langle{\bf u},{\bf v}\rangle|^2$. Consequently, (\ref{ChSWIn}) can written as  $\|{\bf u}\|^2\|{\bf v}\|^2-|\langle{\bf u},{\bf v}\rangle|^2\in \mathbb{D}_+$. Since $\langle{\bf u},{\bf v}\rangle$ is appreciable, by Proposition \ref{P6.5} (c)-(f), we obtain the desired result. 
\end{proof}

The following singular value decomposition (SVD) of dual quaternion matrices can be founded in \cite{QL21}.
\begin{Thm}\label{SVD-DQM}\cite{QL21}
For given $A\in \mathbb{DQ}^{m\times n}$, there exists a dual quaternion unitary matrix $U\in \mathbb{DQ}^{m\times m}$ and a dual quaternion unitary matrix $V\in \mathbb{DQ}^{n\times n}$, such that
\begin{equation}\label{SVDDQMEQ}
A=U\left[\begin{array}{cc}\Sigma_t&O\\
O&O
\end{array} \right]_{m\times n}V^*,
\end{equation}
where $\Sigma_t\in \mathbb{D}^{t\times t}$ is a diagonal matrix, taking the form
$\Sigma_t={\rm diag} (\mu_1,\ldots, \mu_r,\ldots,\mu_t)$, $r \leq t\leq {\min}\{m, n\}$, $\mu_1\geq \mu_2\geq\ldots\geq\mu_r$ are positive appreciable dual numbers, and $\mu_{r+1}\geq \mu_{r+2}\geq\ldots\geq\mu_t$ are positive infinitesimal dual numbers. Counting possible multiplicities of the diagonal entries, the form $\Sigma_t$ is unique.
\end{Thm}

\begin{ReK}
In \cite{QL21}, the positive dual numbers $\mu_1, \ldots,\mu_r,\ldots,\mu_t$ in  Theorem \ref{SVD-DQM} are called the nonzero singular values of $A$. Since $A{\bf x}={\bf 0}$ has only a unique zero solution, if and only if $BAC{\bf y}={\bf 0}$ has only  a unique zero solution, for any nonsingular matrices $B\in \mathbb{DQ}^{m\times m}$ and $C\in \mathbb{DQ}^{n\times n}$, we know that, if ${\rm rank}(A)=n$, then all singular values of $A$ are appreciable; Since $A{\bf x}={\bf 0}$ has only infinitesimal solutions, if and only if $BAC{\bf y}={\bf 0}$ has only infinitesimal solutions, for any nonsingular matrices $B\in \mathbb{DQ}^{m\times m}$ and $C\in \mathbb{DQ}^{n\times n}$, we know that, if ${\rm rank}_w(A)=n$, then all singular values of $A$ are nonzero.
\end{ReK}

The following theorem extends a well-known Fan-Hoffman inequality for complex matrices \cite{FH55} to the dual quaternion matrices.

\begin{Thm}
Let $A\in \mathbb{DQ}^{m\times m}$, and $H=(A^*+A)/2$. Let $\sigma_1(A)\geq\sigma_2(A)\geq\ldots\geq\sigma_m(A)$ be the singular values of $A$, and let $\lambda_1(H)\geq\lambda_2(H)\geq\ldots\geq\lambda_m(H)$ be the eigenvalues  of $H$. Then it holds that
$$
\lambda_k(H)\leq\sigma_k(A), ~~~k=1,2,\ldots,m.
$$
\end{Thm}

\begin{proof}
By Theorem \ref{SVD-DQM}, there exist a dual quaternion unitary matrix $U\in \mathbb{DQ}^{m\times m}$ and a dual quaternion unitary matrix $V\in \mathbb{DQ}^{n\times n}$, such that (\ref{SVDDQMEQ}) holds. It is easy to see that, by (\ref{SVDDQMEQ}), we have
\begin{equation}\label{SVDDQMEQ-1}
A=PQ,
\end{equation}
where $P=UV^*$ and $Q=V\left[\begin{array}{cc}\Sigma_t&O\\
O&O
\end{array} \right]V^*$. Notice that $P$ is unitary, and $Q$ is a dual quaternion Hermitian matrix. Moreover, since $H$ is Hermitian, ${\bf x}^*H{\bf x}\in \mathbb{D}$ for any ${\bf x}\in \mathbb{DQ}^m\backslash[{\bf 0}]$. Hence, it holds that
$$
{\bf x}^*H{\bf x}={\rm Re}({\bf x}^*A{\bf x})={\rm Re}({\bf x}^*PQ{\bf x})\leq |{\bf x}^*PQ{\bf x}|,
$$
where the inequality is due to Proposition \ref{ReDQqin}. Consequently, we have
$$
{\bf x}^*H{\bf x}\leq |{\bf x}^*PQ{\bf x}|=|\langle Q{\bf x}, P^*{\bf x}\rangle|\leq \|Q{\bf x}\|\|P^*{\bf x}\|=({\bf x}^*Q^2{\bf x})^{\frac{1}{2}}\|{\bf x}\|,
$$
where the second inequality comes from Proposition \ref{Ch-SW-In} (Cauchy-Schwarz inequality), and the last equality is due to the fact that $Q$ and $P$ are Hermitian and unitary respectively. Hence, we have
$$
\|{\bf x}\|^{-2}({\bf x}^*H{\bf x})\leq\left(\|{\bf x}\|^{-2}({\bf x}^*Q^2{\bf x})\right)^{\frac{1}{2}}.
$$
Consequently, by Theorem \ref{Lemma3}, for $k=,1,2\ldots,m$,
$$
\begin{array}{lll}
\lambda_k(H)&=&\displaystyle\min_{B\in \mathbb{DQ}^{m\times (k-1)}}\max_{{\bf x}\in N(B^*)\backslash[{\bf 0}]}\|{\bf x}\|^{-2}({\bf x}^*H{\bf x})\\
&\leq&\displaystyle\min_{B\in \mathbb{DQ}^{m\times (k-1)}}\max_{{\bf x}\in N(B^*)\backslash[{\bf 0}]} \left(\|{\bf x}\|^{-2}({\bf x}^*Q^2{\bf x})\right)^{\frac{1}{2}}\\
&=&(\lambda_k(Q^2))^{\frac{1}{2}}\\
&=&(\lambda_k(A^*A))^{\frac{1}{2}}\\
&=&\sigma_k(A).
\end{array}$$
\end{proof}

\section{Generalized inverses of dual quaternion matrices}\label{GenInverseDQM}
In many applications, the involved dual quaternion matrix is singular or rectangular, hence, its inverse matrix does not exist. In this case, we should study generalized inverses of the involved dual quaternion matrix.  
We begin with introducing the following definition, which is similar to the concept of the generalized inverse for matrices in \cite{CW17,WWQ03}.

\begin{Def}\label{MPGIDEF}
For given $A\in \mathbb{DQ}^{m\times n}$. We say $G\in \mathbb{DQ}^{n\times m}$ is a dual quaternion generalized inverse of $A$, if it satisfies
\begin{equation*}\label{MPGI-4EQ}
(1)~AGA=A,~~(2)~GAG=G,~~(3)~(AG)^*=AG,~~{\rm and}~~(4)~(GA)^*=GA.
\end{equation*}
\end{Def}
From Definition \ref{MPGIDEF}, it is easy to see that, for a given $A\in \mathbb{DQ}^{m\times n}$, if its dual quaternion generalized inverse exists, then it is unique, denoted by $A^\dag$. It is well-known that the generalized inverses of complex matrices and quaternion matrices must exist \cite{WLZZ18}, which have unified representations and characterizations \cite{W98,WW03}. However, unlike the general complex and quaternion matrices,  for a given $A\in \mathbb{DQ}^{m\times n}$, its generalized inverse does not necessarily exist. For example, for $A=B\epsilon$ with $B\in \mathbb{Q}^{m\times n}\backslash\{O\}$, it is easy to verify that $AXA=O$ for any $X\in \mathbb{DQ}^{m\times n}$, which implies that the generalized inverse of $A$ must not exist. In this section, we focus on exploring the conditions for the existence of dual quaternion generalized inverses. In what follows, similar to quaternion matrices, when we say, for example, that the matrix $G$ is a $\{i,j,k\}$-dual quaternion generalized inverse of the matrix $A$, we mean that the dual quaternion generalized inverse $G$ satisfies the $i$th, $j$th, and $k$th conditions in Definition \ref{MPGIDEF}.
\begin{Prop}\label{DQ2I}
Let $A=A_{\rm st}+A_{\rm I}\epsilon\in \mathbb{DQ}^{m\times n}$ with $A_{\rm st}\neq O$. Then $X:=A_{\rm st}^\dag-A_{\rm st}^\dag A_{\rm I}A_{\rm st}^\dag\epsilon$ is a $\{2\}$-dual quaternion generalized inverse of $A$, where $A_{\rm st}^\dag$ is the generalized inverse of $A_{\rm st}$.
\end{Prop}
\begin{proof}
Since $A_{\rm st}\in \mathbb{Q}^{m\times n}\backslash\{O\}$, the generalized inverse $A_{\rm st}^\dag$ of $A_{\rm st}$ must exist. Moreover, we have
$$
\begin{array}{lll}
XAX&=&(A_{\rm st}^\dag-A_{\rm st}^\dag A_{\rm I}A_{\rm st}^\dag\epsilon)(A_{\rm st}+A_{\rm I}\epsilon)(A_{\rm st}^\dag-A_{\rm st}^\dag A_{\rm I}A_{\rm st}^\dag\epsilon)\\
&=&\left((A_{\rm st}^\dag A_{\rm st}+(A_{\rm st}^\dag A_{\rm I} -A_{\rm st}^\dag A_{\rm I}A_{\rm st}^\dag A_{\rm st})\epsilon\right)(A_{\rm st}^\dag-A_{\rm st}^\dag A_{\rm I}A_{\rm st}^\dag\epsilon)\\
&=&A_{\rm st}^\dag A_{\rm st}A_{\rm st}^\dag+(A_{\rm st}^\dag A_{\rm I}A_{\rm st}^\dag -A_{\rm st}^\dag A_{\rm I}A_{\rm st}^\dag A_{\rm st}A_{\rm st}^\dag-A_{\rm st}^\dag A_{\rm st}A_{\rm st}^\dag A_{\rm I}A_{\rm st}^\dag)\epsilon\\
&=&A_{\rm st}^\dag -A_{\rm st}^\dag A_{\rm I}A_{\rm st}^\dag \epsilon\\
&=&X,
\end{array}
$$
where the last equality comes from $A_{\rm st}^\dag A_{\rm st}A_{\rm st}^\dag =A_{\rm st}^\dag$ since $A_{\rm st}^\dag$ is the generalized inverse of  $A_{\rm st}$. Hence, $X$ satisfies the condition (2) in Definition \ref{MPGI-4EQ}. 
\end{proof}

Generally speaking, the dual quaternion matrix $X$ in Proposition \ref{DQ2I} does not satisfy other three conditions in Definition \ref{MPGI-4EQ}.
Even in the case of dual number matrices, studying its various dual generalized inverses is much more difficult than that of complex matrices, see \cite{FPU18}. Now we try to present necessary and sufficient conditions for the dual quaternion matrix $X$ in Proposition \ref{DQ2I} to be the $\{1\}$-, $\{3\}$- and $\{4\}$-dual quaternion generalized inverses of $A$ respectively.

\begin{Prop}\label{NScondition-DQ1I}
Let $A\in\mathbb{DQ}^{m\times n}$. Then a necessary and sufficient condition for a matrix $X$ given in Proposition \ref{DQ2I} to be a $\{1\}$-dual quaternion generalized inverse of $A$ is
$$
(I_m-A_{\rm st} A_{\rm st}^\dag)A_{\rm I}(I_n-A_{\rm st}^\dag A_{\rm st})=O.
$$
\end{Prop}
\begin{proof}
Since $X=A_{\rm st}^\dag-A_{\rm st}^\dag A_{\rm I}A_{\rm st}^\dag\epsilon$, it is not difficult to know that
$$
AXA=A_{\rm st}A_{\rm st}^\dag A_{\rm st}+(A_{\rm I}A_{\rm st}^\dag A_{\rm st}-A_{\rm st}A_{\rm st}^\dag A_{\rm I}A_{\rm st}^\dag A_{\rm st}+A_{\rm st}A_{\rm st}^\dag A_{\rm I})\epsilon.
$$
Since $A_{\rm st}A_{\rm st}^\dag A_{\rm st}=A_{\rm st}$, a necessary and sufficient condition for $X$ satisfying condition (1) in Definition \ref{MPGIDEF} is
$$
A_{\rm I}A_{\rm st}^\dag A_{\rm st}-A_{\rm st}A_{\rm st}^\dag A_{\rm I}A_{\rm st}^\dag A_{\rm st}+A_{\rm st}A_{\rm st}^\dag A_{\rm I}=A_{\rm I},
$$
which can be rewritten as $(I_m-A_{\rm st} A_{\rm st}^\dag)A_{\rm I}(I_n-A_{\rm st}^\dag A_{\rm st})=O$.
\end{proof}

\begin{Prop}\label{NScondition-DQ3I}
Let $A\in\mathbb{DQ}^{m\times n}$. Then a necessary and sufficient condition for a matrix $X$ given in Proposition \ref{DQ2I} to be a $\{3\}$-dual quaternion generalized inverse of $A$ is
$$
(I_m-A_{\rm st} A_{\rm st}^\dag)A_{\rm I}A_{\rm st}^\dag=(A_{\rm st}^*)^\dag A_{\rm I}^*(I_m-A_{\rm st} A_{\rm st}^\dag),
$$
i.e., the dual quaternion matrix $(I_m-A_{\rm st} A_{\rm st}^\dag)A_{\rm I}A_{\rm st}^\dag$ is Hermitian.
\end{Prop}
\begin{proof}
Since $X=A_{\rm st}^\dag-A_{\rm st}^\dag A_{\rm I}A_{\rm st}^\dag\epsilon$, we have
$$
AX=A_{\rm st}A_{\rm st}^\dag+(A_{\rm I}A_{\rm st}^\dag -A_{\rm st}A_{\rm st}^\dag A_{\rm I}A_{\rm st}^\dag )\epsilon
$$
and
$$
(AX)^*=(A_{\rm st}A_{\rm st}^\dag)^*+\left(A_{\rm I}A_{\rm st}^\dag -A_{\rm st}A_{\rm st}^\dag A_{\rm I}A_{\rm st}^\dag \right)^*\epsilon=A_{\rm st}A_{\rm st}^\dag+\left((A_{\rm st}^\dag)^* A_{\rm I}^*-(A_{\rm st}^\dag)^* A_{\rm I}^*A_{\rm st}A_{\rm st}^\dag \right)\epsilon,
$$
where the last equality is due to the fact that $A_{\rm st}^\dag$ is the generalized inverse of $A_{\rm st}$. Consequently, we know that a necessary and sufficient condition for $X$ satisfying $(AX)^*=AX$ is
$$
(A_{\rm st}^\dag)^* A_{\rm I}^*-(A_{\rm st}^\dag)^* A_{\rm I}^*A_{\rm st}A_{\rm st}^\dag =A_{\rm I}A_{\rm st}^\dag -A_{\rm st}A_{\rm st}^\dag A_{\rm I}A_{\rm st}^\dag,
$$
which can be rewritten as $(I_m-A_{\rm st} A_{\rm st}^\dag)A_{\rm I}A_{\rm st}^\dag=(A_{\rm st}^*)^\dag A_{\rm I}^*(I_m-A_{\rm st} A_{\rm st}^\dag)
$, since $(A_{\rm st}^\dag)^*=(A_{\rm st}^*)^\dag$.
\end{proof}

\begin{Prop}\label{NScondition-DQ4I}
Let $A\in\mathbb{DQ}^{m\times n}$. Then a necessary and sufficient condition for a matrix $X$ given in Proposition \ref{DQ2I} to be a $\{4\}$-dual quaternion generalized inverse of $A$ is
$$
A_{\rm st}^\dag A_{\rm I}(I_n-A_{\rm st}^\dag A_{\rm st})=(I_n-A_{\rm st}^\dag A_{\rm st})A_{\rm I}^* (A_{\rm st}^*)^\dag,
$$
i.e., the dual quaternion matrix $A_{\rm st}^\dag A_{\rm I}(I_n-A_{\rm st}^\dag A_{\rm st})$ is Hermitian.
\end{Prop}
\begin{proof}
It can be proved by a method similar to the proof of Proposition \ref{NScondition-DQ3I}
\end{proof}

\begin{ReK}
When ${\rm rank}(A)=n$, we have ${\rm rank}(A_{\rm st})=n$ from argument in Section \ref{RighIN-Rank}. Consequently, by generalized inverse theory on quaternion matrices \cite{WLZZ18}, it holds that $A_{\rm st}^\dag A_{\rm st}=I_n$. By Propositions \ref{NScondition-DQ1I} and \ref{NScondition-DQ4I}, we know that the matrix $X$ in Proposition \ref{DQ2I} is exactly the $\{1,2,4\}$-dual quaternion generalized inverse of $A$.
\end{ReK}

Below we discuss a special case of dual quaternion matrices, where the dual quaternion generalized inverses of the dual quaternion matrices must exist and has a clear expression. 


\begin{Prop}
Let $A\in\mathbb{DQ}^{m\times n}$. Suppose that its singular value decomposition is given by (\ref{SVDDQMEQ}), and all nonzero singular values of $A$ are appreciable, i.e., $t=r$ in (\ref{SVDDQMEQ}). Then
\begin{equation}
A^\dag=V\left[\begin{array}{cc}
\Sigma_t^{-1}&O\\
O&O
\end{array}\right]_{n\times m} U^*,
\end{equation}
where $\Sigma_t^{-1}={\rm diag}(\mu_1^{-1},\mu_2^{-1},\ldots,\mu_t^{-1})$.
\end{Prop}
\begin{proof}
By a direct verification, we can know that all conditions (1)-(4) in Definition \ref{MPGI-4EQ} are satisfied. Hence, we obtain the desired conclusion.
\end{proof}

\section{Conclusion}\label{Conclusion}
In this paper, we introduced three different right linear independency concepts for dual quaternion vectors, and studied some related basic properties of dual quaternion vectors and dual quaternion matrices. We presented a minimax principle for eigenvalues of dual quaternion Hermitian matrices as well as a dual quaternion matrix version of the well-known Fan-Hoffman inequality of complex matrices. Finally, by introducing the concept of Moore-Penrose type generalized inverses of dual quaternion matrices, we presented necessary and sufficient conditions for a dual quaternion matrix to be one of four types of generalized inverses of another dual quaternion matrix. This is a new area of applied mathematics. More problems are worth exploring, such as the physical meaning of the minimax principle of dual quaternion Hermitian matrices and the application of dual quaternion matrices in multi-agent formation control.

\end{document}